\documentclass[a4paper,11pt]{article}
\usepackage[latin1]{inputenc}
\usepackage[english]{babel}
\usepackage{amsmath}
\usepackage{amsfonts}
\usepackage{amssymb}
\usepackage{epsfig}
\usepackage{amsopn}
\usepackage{amsthm}
\usepackage{color}
\usepackage{graphicx}
\usepackage{subfigure}
\usepackage{enumerate}
\newtheorem{theorem}{Theorem}[section]

\newtheorem{lemma}[theorem]{Lemma}
\newtheorem{proposition}[theorem]{Proposition}
\newtheorem{definition}[theorem]{Definition}

\newtheorem*{theorem*}{Theorem}
\newtheorem*{lemma*}{Lemma}
\newtheorem*{remark*}{Remark}
\newtheorem*{definition*}{Definition}
\newtheorem*{proposition*}{Proposition}
\newtheorem*{corollary*}{Corollary}
\numberwithin{equation}{section}
\newcommand{\real}{\mathbb{R}}

\def\qed{\,\unskip\kern 6pt \penalty 500
\raise -2pt\hbox{\vrule \vbox to8pt{\hrule width 6pt
\vfill\hrule}\vrule}\par}
\definecolor{darkblue}{rgb}{0.05, .05, .65}
\definecolor{darkgreen}{rgb}{0.1, .65, .1}
\definecolor{darkred}{rgb}{0.8,0,0}
\newcommand{\beqn}{\begin{equation}}
\newcommand{\eeqn}{\end{equation}}
\newcommand{\bear}{\begin{eqnarray}}
\newcommand{\eear}{\end{eqnarray}}
\newcommand{\bean}{\begin{eqnarray*}}
\newcommand{\eean}{\end{eqnarray*}}
\begin{document}

\title{\huge \bf Blow up profiles for a reaction-diffusion equation with critical weighted reaction}

\author{
\Large Razvan Gabriel Iagar\,\footnote{Departamento de Matem\'{a}tica
Aplicada, Ciencia e Ingenieria de los Materiales y Tecnologia
Electr\'onica, Universidad Rey Juan Carlos, M\'{o}stoles,
28933, Madrid, Spain, \textit{e-mail:}
razvan.iagar@urjc.es}, \footnote{Institute of Mathematics of the
Romanian Academy, P.O. Box 1-764, RO-014700, Bucharest, Romania.}
\\[4pt] \Large Ariel S\'{a}nchez,\footnote{Departamento de Matem\'{a}tica
Aplicada, Ciencia e Ingenieria de los Materiales y Tecnologia
Electr\'onica, Universidad Rey Juan Carlos, M\'{o}stoles,
28933, Madrid, Spain, \textit{e-mail:} ariel.sanchez@urjc.es}\\
[4pt] }
\date{}
\maketitle

\begin{abstract}
We classify the blow up self-similar profiles for the following reaction-diffusion equation with weighted reaction
$$
u_t=(u^m)_{xx} + |x|^{\sigma}u^m,
$$
posed for $(x,t)\in\real\times(0,T)$, with $m>1$ and $\sigma>0$. In strong contrast with the well-studied equation without the weight (that is $\sigma=0$), on the one hand we show that for $\sigma>0$ sufficiently small there exist \emph{multiple self-similar profiles with interface} at a finite point, more precisely, given any positive integer $k$, there exists $\delta_k>0$ such that for $\sigma\in(0,\delta_k)$, there are at least $k$ different blow up profiles with compact support and interface at a positive point. On the other hand, we also show that for $\sigma$ sufficiently large, the blow up self-similar profiles with interface \emph{cease to exist}. This unexpected balance between existence of multiple solutions and non-existence of any, when $\sigma>0$ increases, is due to the effect of the presence of the weight $|x|^{\sigma}$, whose influence is the main goal of our study. We also show that for any $\sigma>0$, there are no blow up profiles supported in the whole space, that is with $u(x,t)>0$ for any $x\in\real$ and $t\in(0,T)$.
\end{abstract}

\


\smallskip


\section{Introduction}

The aim of this paper is to contribute to the study of the finite time blow up phenomenon for the quasilinear reaction-diffusion equation with weighted reaction term
\begin{equation}\label{eq1}
u_t=(u^m)_{xx}+|x|^{\sigma}u^m, \quad u=u(x,t), \quad (x,t)\in\real\times(0,T),
\end{equation}
with $m>1$ and $\sigma>0$, where the notations $u_t$ and $u_{xx}$ in \eqref{eq1} indicate, as usual, partial derivatives with respect to the corresponding time or space variables. It is already well established (see for example the books \cite{S4, QS} for homogeneous reaction terms or the two recent papers of the authors \cite{IS1, IS2} for weighted reaction terms) that finite time blow up is expected to occur for solutions to \eqref{eq1}, that means, there exists $T\in(0,\infty)$ such that $u(t)\in L^{\infty}(\real)$ for any $t\in(0,T)$ but $u(T)\not\in L^{\infty}(\real)$. The time $T\in(0,\infty)$ satisfying this property is called the blow up time of the solution $u$. Here and in the sequel we will denote for simplicity by $u(t)$ the map $x\mapsto u(x,t)$ for a fixed time $t\geq0$.

The finite time blow up phenomenon for the reaction-diffusion equation
\begin{equation}\label{eq.hom.gen}
u_t=\Delta u^m + u^p,
\end{equation}
with $m>1$, $p>1$, has been well investigated by now at least in dimension $N=1$ (\cite{S4, GV97}), although in dimensions $N>1$ there remain several interesting open problems concerning the uniqueness of blow up self-similar profiles and with the convergence to them. In particular, when $m=p$ and in spatial dimension one, that is,
\begin{equation}\label{eq.hom}
u_t=(u^m)_{xx}+u^m,
\end{equation}
which corresponds to \eqref{eq1} for $\sigma=0$, it is shown that there exists a unique (up to translation with respect to the $x$ variable) blow up profile, which is explicit
\begin{equation}\label{sol.hom}
u(x,t)=(T-t)^{1/(m-1)}F(|x|), \quad F(\xi)=\left[\frac{2m}{(m-1)(m+1)}\cos^2\left(\frac{\pi\xi}{L_S}\right)\right]_+^{1/(m-1)},
\end{equation}
where $L_S=2\pi\sqrt{m}/(m-1)$. It is then proved that this explicit solution, whose support is fixed, is the asymptotic profile near blow up for general solutions and that general solutions to \eqref{eq.hom} present localization of the support \cite[Chapter 4, Sections 4 and 5]{S4}. Moreover, in the same chapter these results are generalized to any dimension $N>1$, where again existence and uniqueness of the blow up profile are established, the form of it being no longer explicit.

Concerning reaction-diffusion equations of the more general form
\begin{equation}\label{eq.nohom.gen}
u_t=\Delta u^m+|x|^{\sigma}u^p,
\end{equation}
a number of results have been established, mostly in the last three decades. A big amount of these previous results focus on studying when blow up in finite time occurs, that is, for which exponents $p>1$ and for which initial conditions $u_0(x)=u(x,0)$, $x\in\real^N$. Among them, we quote a series of works devoted to establishing the Fujita exponent and to studying the "life-span" of solutions (that is, understanding how the blow up time of a one-parameter family of solutions changes with respect to the parameter) for the semilinear case $m=1$, \cite{BL89, BK87, Pi97, Pi98}, where more general weights than the pure powers $|x|^{\sigma}$ are considered. More recently, Suzuki in \cite{Su02} extends the Fujita exponent to the quasilinear case $m>1$ and establishes sufficient conditions on the tails of the initial data $u_0(x)$ as $|x|\to\infty$ for the finite time blow up to occur, but restricting himself to the range of exponents $p>m$. Furthermore, Andreucci and Tedeev establish the blow up rate in \cite{AT05}, but again restricting the analysis to the range of exponents $m<p<m+N/2$ and $0<\sigma\leq N(p-m)/m$. Coming back to the semilinear case $m=1$, the possibility of $x=0$ to be a blow up point is studied in the series of papers \cite{GLS, GS11, GLS13}, and it is shown that in general, most solutions cannot blow up at $x=0$ (something formally expected due to the specific form of the weighted reaction), but the origin can be still a blow up point in some very specific cases. A different direction of study where interesting results were obtained was studying equations of the form
$$
u_t=\Delta u^m+a(x)u^p,
$$
where $a(x)$ is a compactly supported function, first in dimension $N=1$ \cite{FdPV06} and later in dimension $N>1$, see \cite{BZZ11, KWZ11, Liang12, FdP18}, where also the fast diffusion case $m<1$ is considered. In particular, it is proved in \cite{FdP18} that for $N\geq2$ and $p=m$, solutions present grow up instead of blow up. For the one-dimensional case, the work \cite{FdPV06} goes into a deeper study of the blow up phenomenon, establishing blow up rates, sets and profiles.

This is why, the main goal of the larger project started by the authors in their previous works \cite{IS1, IS2} is to understand how an unbounded weight on the reaction term (typically a pure positive power $|x|^{\sigma}$) affects the blow up behavior in all its aspects (profiles, rates, blow up points, asymptotic behavior near the blow up time). The unboundedness of the weights is highly relevant, as already shown in the quoted works, since the reaction becomes much stronger at points where $|x|$ is large. The authors devote \cite{IS1} to the study of the blow up profiles for the reaction exponent $p=1$, proving that the weight $|x|^{\sigma}$ leads to finite time blow up even in this case (while in the non-weighted equation solutions are always global). The second paper \cite{IS2} is dedicated to the study and classification of the blow up profiles for $m>1$ and $1<p<m$, and in both works we show that there is a strong influence of the "strength" of the weight: generically, there exists a $\sigma^*\in(0,\infty)$ critical such that, for $\sigma<\sigma_*$ blow up is global and the blow up profiles have a form, while for $\sigma>\sigma_*$ blow up occurs \emph{only at the space infinity} and the blow up profiles have a very different form with respect to the ones established when $\sigma>0$ was small enough. In the present paper we deal with the case when both exponents are equal, $m=p$, and we show that the results concerning blow up profiles strongly depart both with respect to the non-weighted case $\sigma=0$ (described in \cite[Chapter 4]{S4}) and to the case $1<p<m$ with $\sigma>0$ (described in \cite{IS2}). We explain below how, by introducing our main results.

\medskip

\noindent \textbf{Main results.} It is a well established fact by now that special solutions, usually in self-similar form, enclose very important information about the qualitative properties of solutions to diffusion equations and usually are "optimal" solutions both in a priori estimates for general solutions and as patterns that generic solutions approach asymptotically (either as $t\to\infty$ in the case of global solutions, or as $t\to T$ if finite time blow up occurs). Thus, in our case they are likely to be \emph{blow up profiles} for Eq. \eqref{eq1}. This is why we are strongly interested in finding and classifying all the self-similar blow up solutions to Eq. \eqref{eq1}, that is, solutions to Eq. \eqref{eq1} having the particular form
\begin{equation}\label{SSform}
u(x,t)=(T-t)^{-\alpha}f(\xi), \quad \xi=|x|(T-t)^{\beta},
\end{equation}
where $T\in(0,\infty)$ is the finite blow up time and $\alpha>0$, $\beta\in\real$ exponents to be determined. Replacing the form \eqref{SSform} in \eqref{eq1}, we readily find that $\alpha=1/(m-1)$, $\beta=0$ and the self-similar profile $f$ (which is in fact of separate variables since $\beta=0$) solves the non-autonomous differential equation
\begin{equation}\label{SSODE}
(f^m)''(\xi)-\frac{1}{m-1}f(\xi)+\xi^{\sigma}f^{m}(\xi)=0, \quad \xi\in[0,\infty).
\end{equation}
We perform in the sequel a deep study of the previous ODE. We thus introduce the type of profiles $f$ we look for, similar to \cite{IS2}.
\begin{definition}\label{def1}
We say that $f$ solution to \eqref{SSODE} is a \textbf{good profile}
if it fulfills one of the following two properties related to its
behavior at $\xi=0$:

\indent (P1) $f(0)=a>0$, $f'(0)=0$.

\indent (P2) $f(0)=0$, $(f^m)'(0)=0$.

A good profile $f$ is called a \textbf{good profile with interface}
at some point $\eta\in(0,\infty)$ if
$$
f(\eta)=0, \qquad (f^m)'(\eta)=0, \qquad f>0 \ {\rm on} \
(\eta-\delta,\eta), \ {\rm for \ some \ } \delta>0.
$$
\end{definition}
In our previous works \cite{IS1, IS2} we have proved that good profiles with interface exist for any $\sigma>0$, and conjectured that for any $\sigma>0$, there is a unique good profile (a conjecture we could not yet prove but which it is strongly supported by numerical experiments). Moreover, depending on how big is $\sigma>0$, good profiles may satisfy either assumption (P1) (if $\sigma>0$ is rather small) or assumption (P2) in Definition \ref{def1} (if $\sigma>0$ is large). Surprisingly, the results we state below, concerning our current exponents $m=p$, is in a \emph{striking contrast} to the above ones. We begin by our existence result, which occurs for $\sigma>0$ sufficiently small.
\begin{theorem}[Existence of multiple good profiles with interface for $\sigma$ small]\label{th.exist}
Given any positive integer $k$, there exists $\delta_k>0$ sufficiently small such that for any $\sigma\in(0,\delta_k)$, there exist \textbf{at least $k$ different} good profiles with interface to \eqref{SSODE}. All these profiles satisfy assumption (P1) in Definition \ref{def1}. There is no good profile with interface satisfying property (P2) in Definition \ref{def1}.
\end{theorem}
Let us remark that, if for $1<p<m$ we were not able yet in \cite{IS2} to prove uniqueness of the good profiles (for $\sigma>0$ given), but all the numerical experiments suggest it, in the case under study $p=m$, we show that \emph{uniqueness does not hold true}! More precisely, we shall see that given $k>0$ and $\sigma\in(0,\delta_k)$ as in Theorem \ref{th.exist}, we will find $k$ different profiles, each of them having a different number of local maxima and minima and oscillating first a finite number of times around the explicit hyperbola obtained as the graph of the function
\begin{equation}\label{hyperbola}
f(\xi)=\left(\frac{1}{m-1}\right)^{1/(m-1)}\xi^{-\sigma/(m-1)},
\end{equation}
before at some finite point leaving the oscillation in order to decrease and reach their interface point. This is also contrasting to the case $\sigma=0$ where uniqueness of good profiles with interface is established, the only such profile being explicitely given by \eqref{sol.hom}. We plot in Figure \ref{fig1} a few good profiles with interface, as a visual representation of the multiplicity theorem. Let us notice that the profiles in Figure \ref{fig1} have one, two, three, four, respectively five local maxima.
\begin{figure}[ht!]
  \begin{center}
  \includegraphics[width=10cm,height=7.5cm]{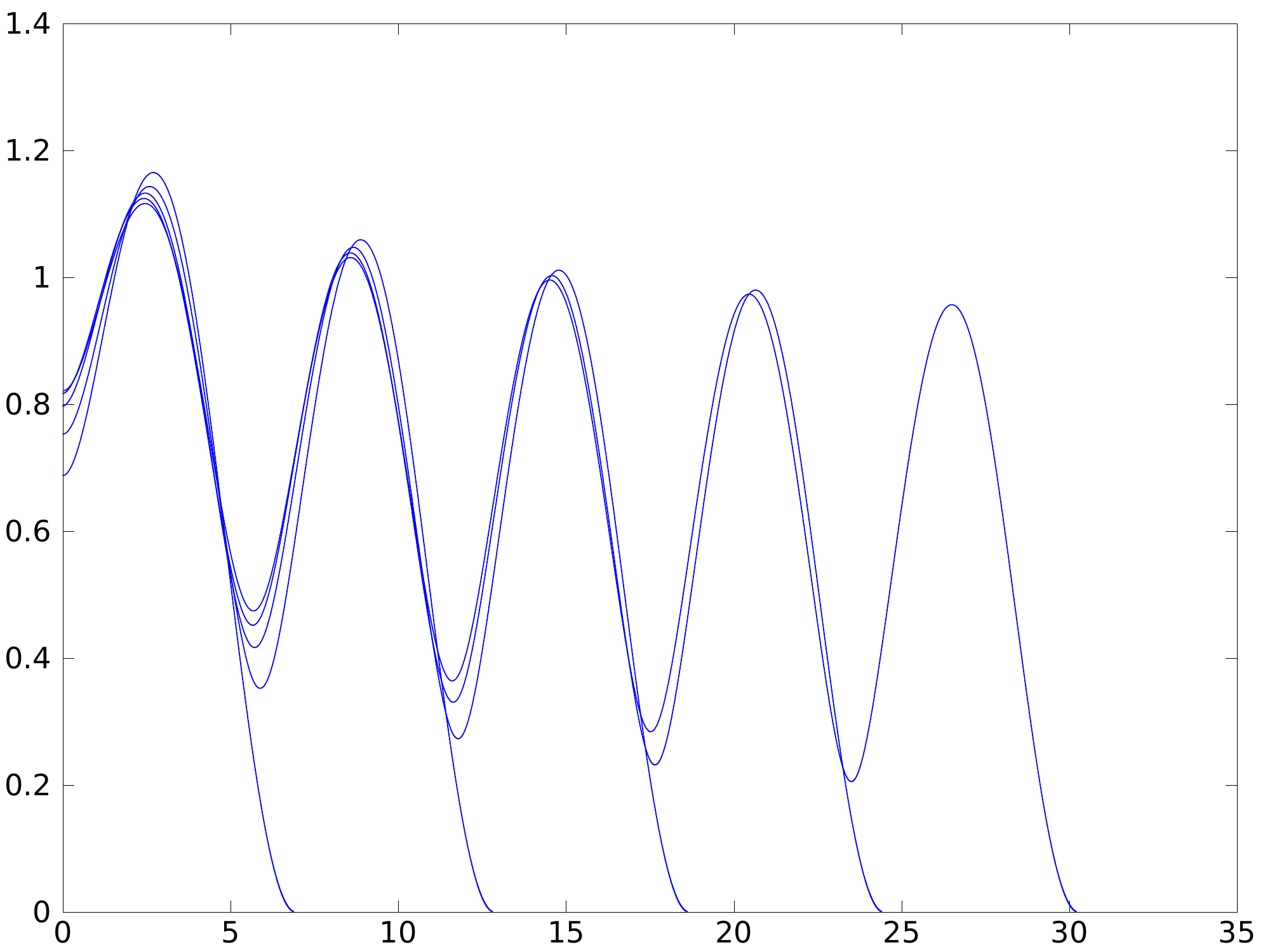}
  \end{center}
  \caption{Some good profiles with interface for $\sigma$ small. Experiment for $m=2$, $p=2$ and $\sigma=0.1$}\label{fig1}
\end{figure}

\medskip

On the other hand, since we are dealing with small values of $\sigma>0$, the fact that all these profiles behave like in the assumption (P1) as $\xi\to0$ may seem not so striking to the reader. However, the next result is strongly contrasting to what is known in the case $1<p<m$.
\begin{theorem}[Non-existence of good profiles with interface for $\sigma$ large]\label{th.nonexist}
There exists $\sigma_0>0$ such that for any $\sigma>\sigma_0$, there are \textbf{no good profiles with interface} to \eqref{SSODE}.
\end{theorem}
This is indeed a big difference with respect to the neighbor case when $p<m$, where we have seen that existence of good profiles with interface is granted for any $\sigma>0$. The deep reason behind this non-existence result is the fact that profiles satisfying assumption (P2) in Definition \ref{def1} will never present an interface behavior, while the profiles with an interface at some positive point will always intersect the vertical axis with negative slope (as it will result from the proof in Section \ref{sec.nonexist}). As a sample, we plot in Figure \ref{fig2} a few profiles with interface, noticing that they always satisfy $f'(0)<0$.
\begin{figure}[ht!]
  \begin{center}
  \includegraphics[width=10cm,height=7.5cm]{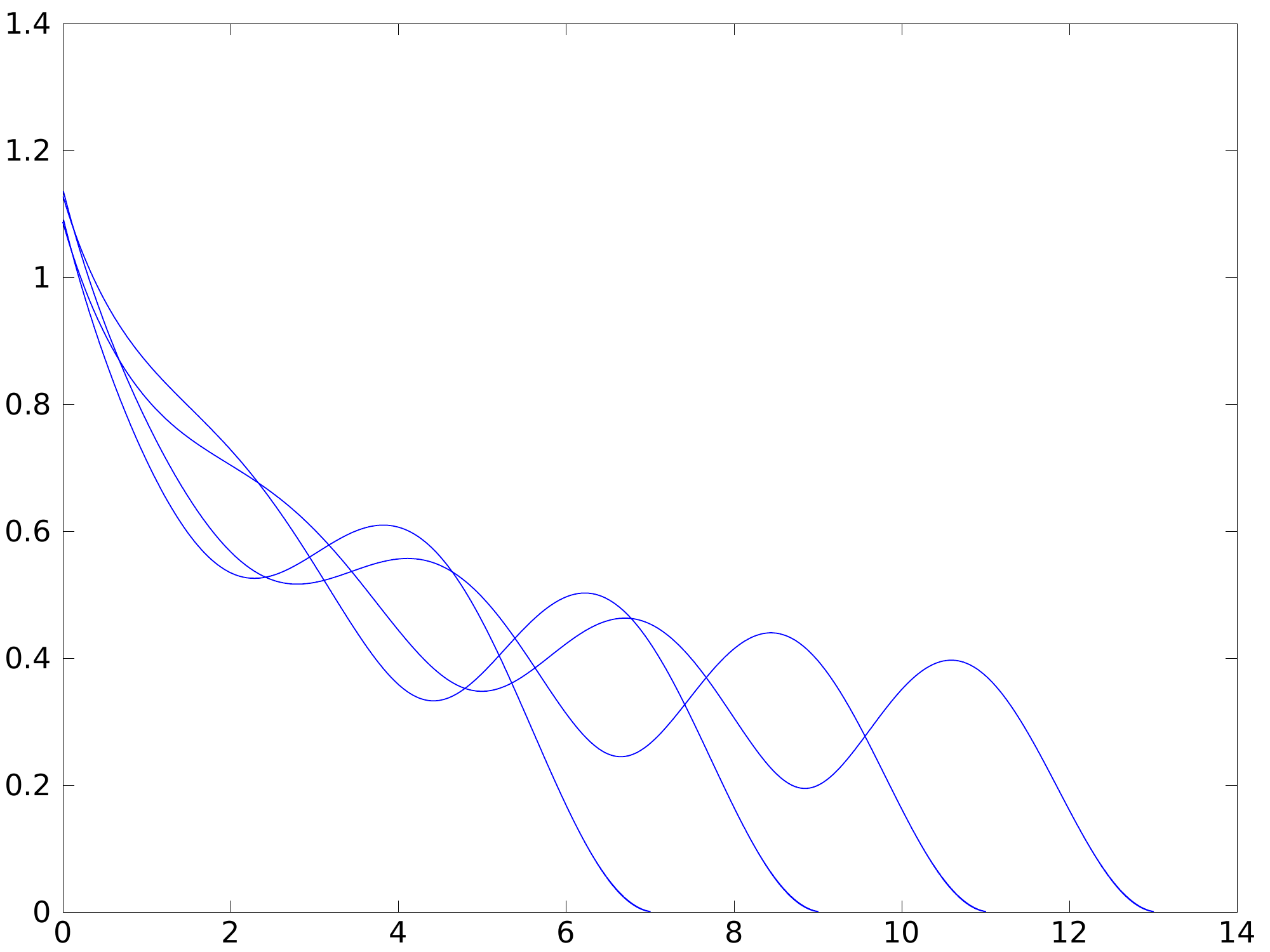}
  \end{center}
  \caption{A few profiles with interface and negative slope at the origin. Experiment for $m=2$, $p=2$ and $\sigma=0.5$} \label{fig2}
\end{figure}

\medskip

Moreover, another striking non-existence result occurs:
\begin{theorem}[Non-existence of positive good profiles]\label{th.nontail}
Given $\sigma>0$, there is no good profile $f$ to \eqref{SSODE} such that $f(\xi)>0$ for any $\xi>0$.
\end{theorem}
We recall that for $\sigma=0$, there exists such a profile, which is in fact the constant profile $f(\xi)=(1/(m-1))^{1/(m-1)}$. The explanation behind the non-existence for $\sigma>0$ is that, as we will make it rigorous in the proofs, the constant profile is replaced by the function \eqref{hyperbola} and that is not a solution to \eqref{SSODE}. Moreover, such profiles with a tail as $\xi\to\infty$ do exist for any $\sigma>0$ if $1\leq p<m$, as shown in \cite{IS1, IS2}, thus, it was rather unexpected for us to discover that they cease to exist when $m=p$. The reasons for this difference will be detailed in Section \ref{sec.bifur}.


\medskip

\noindent \textbf{Techniques of the proofs and organization of the paper.} The proofs of the main results will be based (as we also did in \cite{IS1, IS2}) on the analysis of a phase-space associated to a three-dimensional autonomous system associated to Eq. \eqref{SSODE}. We stress here that the system we use in the present work is different from the ones used in the mentioned previous papers, and not just a particular case of them, thus the analysis of the critical points has to be done in detail again. As the main technique for the global analysis, we will employ again the \emph{backward shooting from an interface point} (see \cite{GP76, QV95} as older papers employing this specific technique, also used by the authors in \cite{IS1, IS2}), combined with the existence of a very convenient separatrix for the phase space in form of an infinite cylinder. The deduction of the system and the local analysis of the hyperbolic critical points is performed in Section \ref{sec.local}. We intentionally leave aside a very complicated non-hyperbolic critical point whose local analysis requires techniques of bifurcation theory and will be performed in Section \ref{sec.bifur}. The non-existence of any connection in the phase space entering that special point implies immediately the non-existence result in Theorem \ref{th.nontail}. The next Section \ref{sec.exist} is devoted to the proof of Theorem \ref{th.exist}; for this proof, using the original phase space is no longer sufficient, and the main technique is based by a clever use of the continuity with respect to the parameter $\sigma$, but in a different autonomous system. The most involved result of the paper, from the point of view of its proof, is the non-existence one. In order to prove Theorem \ref{th.nonexist}, one has to notice first a \emph{partial monotonicity of the profiles} with respect to the set of parameters $(f(0),f'(0))$, and then mix this ingredient with a direct shooting technique (from $x=0$) and the outcome of the global analysis of the phase space in order to show that for $\sigma>0$ large, a profile one shoots with $f(0)=a>0$ and $f'(0)=0$, and another profile one shoots backward from an interface point $\xi_0>0$, can never meet and join into a single one. All these facts are proved in Section \ref{sec.nonexist} which ends the paper.


\section{The phase space. Local analysis of the critical points}\label{sec.local}

This section is rather technical and is devoted to the local analysis of the hyperbolic critical points in a phase space associated to a quadratic dynamical system to which Eq. \eqref{SSODE} can be transformed. More precisely, let us set
\begin{equation}\label{PSchange}
\begin{split}
&X(\eta)=\sqrt{m(m-1)}\xi^{-1}f^{(m-1)/2}(\xi), \ \ Y(\eta)=\frac{2\sqrt{m(m-1)}}{m-1}(f^{(m-1)/2})'(\xi), \\ &Z(\eta)=(m-1)\xi^{\sigma}f^{m-1}(\xi),
\end{split}
\end{equation}
together with the change of independent variable given by
$$
\frac{d\eta}{d\xi}=\frac{1}{\sqrt{m(m-1)}}f^{-(m-1)/2}(\xi),
$$
to transform Eq. \eqref{SSODE} into the following quadratic autonomous system of differential equations
\begin{equation}\label{PSsyst2}
\left\{\begin{array}{ll}\dot{X}=\frac{m-1}{2}XY-X^2,\\
\dot{Y}=-\frac{m+1}{2}Y^2+1-Z,\\
\dot{Z}=Z[(m-1)Y+\sigma X],\end{array}\right.
\end{equation}
where the derivative is taken with respect to the new variable $\eta$. Let us notice that, although some of the equations (the ones for $X$ and $Z$ variables) are almost the same, the system in itself is sensibly different from the corresponding one we used in \cite[Section 2]{IS2} for the range $1<p<m$. To simplify the notation, let us also set $h_0=\sqrt{2/(m+1)}$.

\medskip

\noindent \textbf{Local analysis of the hyperbolic critical points in the plane.} We notice that the system \eqref{PSsyst2} has four critical points in the finite part of the phase plane, these are
$$
P_0=(0,h_0,0), \ \ P_1=(0,-h_0,0), \ \ P_2=\left(\frac{(m-1)h_0}{2},h_0,0\right), \ \ P_3=(0,0,1).
$$
We analyze the local behavior of the orbits in the phase space near these points below, with the exception of $P_3$, which is non-hyperbolic and whose analysis is postponed to the next section.
\begin{lemma}[Analysis of the points $P_0$ and $P_1$]\label{lem.P0P1}
The system in a neighborhood of the critical point $P_0$ has a two-dimensional unstable manifold and a one-dimensional stable manifold. The orbits going out of $P_0$ on the unstable manifold contain profiles such that
\begin{equation}\label{behP0}
f(\xi)\sim\left(\frac{(m-1)h_0}{2\sqrt{m(m-1)}}\xi-K\right)_+^{2/(m-1)}, \quad K>0, \ \ {\rm as} \ \xi\to\xi_0=\frac{2K\sqrt{m(m-1)}}{(m-1)h_0},
\end{equation}
that is, profiles that enter the positive region $f(\xi)>0$ with an interface at a finite point $\xi=\xi_0>0$. On the contrary, the system in a neighborhood of the critical point $P_1$ has a one-dimensional unstable manifold and a two-dimensional stable manifold. The orbits entering $P_1$ on the stable manifold contain profiles such that
\begin{equation}\label{behP1}
f(\xi)\sim\left(K-\frac{(m-1)h_0}{2\sqrt{m(m-1)}}\xi\right)_+^{2/(m-1)}, \quad K>0, \ \ {\rm as} \ \xi\to\xi_0=\frac{2K\sqrt{m(m-1)}}{(m-1)h_0},
\end{equation}
that is, profiles with an interface at a positive point $\xi=\xi_0>0$.
\end{lemma}
\begin{proof}
The linearization of the system \eqref{PSsyst2} in a neighborhood of the critical points $P_0$, respectively $P_1$, has the matrix
$$
M=\left(
         \begin{array}{ccc}
           \pm(m-1)h_0/2 & 0 & 0 \\
           0 & \mp(m+1)h_0 & -1 \\
           0 & 0 & \pm(m-1)h_0 \\
         \end{array}
       \right),$$
with eigenvalues $\lambda_1=\pm(m-1)h_0/2$, $\lambda_2=\mp(m+1)h_0$, $\lambda_3=\pm(m-1)h_0$, where the plus sign corresponds to $P_0$ and the minus sign to $P_1$. Thus, $P_0$ has a two-dimensional unstable manifold, while $P_1$ has a two-dimensional stable manifold. It is easy to check that the orbits entering $P_0$ on the one-dimensional stable manifold, or going out of $P_1$ on the one-dimensional unstable manifold (both corresponding to the eigenvalue $\lambda_2$) are contained in the $Y$ axis. The orbits going out of $P_0$ (respectively entering $P_1$) on the two-dimensional unstable manifold (respectively the two-dimensional stable manifold) contain profiles such that $X\to0$, $Y\to\pm h_0$ and $Z\to0$. We deduce using \eqref{PSchange} that on the one hand this gives
\begin{equation}\label{interm22}
(f^{(m-1)/2})'(\xi)\sim\pm\frac{(m-1)h_0}{2\sqrt{m(m-1)}},
\end{equation}
and on the other hand shows that \eqref{interm22} holds for $\xi\to\xi_0>0$ finite. Indeed, assuming for contradiction that \eqref{interm22} holds true with $\xi\to\infty$, since $Z(\xi)\to0$, it follows that $f(\xi)\to0$ as $\xi\to\infty$, and on the contrary since $Y(\xi)\to\pm h_0$, this implies that $(f^{(m-1)/2})'(\xi)\to\pm K$ as $\xi\to\infty$ for some constant $K>0$, which is a contradiction. And if we assume for contradiction that \eqref{interm22} holds true with $\xi\to0$, since $X(\xi)\to 0$ it follows that $f^{(m-1)/2}(\xi)\to0$ as $\xi\to0$. Moreover, since $Y(\xi)\to\pm h_0$ it follows that $(f^{(m-1)/2})'(\xi)\to\pm K$ as $\xi\to0$, for some constant $K>0$. Letting $h(\xi)=f^{(m-1)/2}(\xi)$, we obtain that $h'(\xi)\to\pm K$ and $\xi^{-1}h(\xi)\to0$ as $\xi\to0$, which leads to a contradiction. We thus discard the possibility that $\xi\to0$ and $\xi\to\infty$ in \eqref{interm22}, and we then get the behavior \eqref{behP0} by integration in \eqref{interm22} when working with the plus sign, as $\xi\to\xi_0\in(0,\infty)$ (respectively \eqref{behP1} when working with the minus sign).
\end{proof}
\begin{lemma}[Analysis of the point $P_2$]\label{lem.P2}
The system in a neighborhood of the critical point $P_2$ has a two-dimensional stable manifold and a one-dimensional unstable manifold. The stable manifold is contained in the invariant plane $\{Z=0\}$. There exists a unique orbit going out of $P_2$, containing profiles such that
\begin{equation}\label{behP2}
f(0)=0, \quad f(\xi)\sim\left[\frac{m-1}{2m(m+1)}\right]^{1/(m-1)}\xi^{2/(m-1)}, \ \ {\rm as} \ \xi\to0.
\end{equation}
\end{lemma}
\begin{proof}
The linearization of the system \eqref{PSsyst2} near the critical point $P_2$ has the matrix
$$
M(P_2)=\left(
         \begin{array}{ccc}
           -(m-1)h_0/2 & (m-1)^2h_0/4 & 0 \\
           0 & -(m+1)h_0 & -1 \\
           0 & 0 & (m-1)(\sigma+2)h_0/2 \\
         \end{array}
       \right),
$$
with eigenvalues $\lambda_1=-(m-1)h_0/2$, $\lambda_2=-(m+1)h_0$ and $\lambda_3=(m-1)(\sigma+2)h_0/2$. It is easy to verify (similarly as in \cite[Lemma 2.3]{IS1}) that the two-dimensional stable manifold is contained in the invariant plane $\{Z=0\}$ and there exists only one orbit going out of $P_2$ into the region $\{Z>0\}$ of the phase space. Taking into account that on this orbit we have $X\to(m-1)h_0/2$, $Y\to h_0$ and $Z\to0$, we deduce from \eqref{PSchange} that this unique orbit contain profiles such that
$$
\lim\limits_{\xi\to0}X(\xi)=\lim\limits_{\xi\to0}\sqrt{m(m-1)}f^{(m-1)/2}\xi^{-1}=\frac{m-1}{2}h_0,
$$
where the fact that $\xi\to0$ follows from an analysis similar to the one performed in the proof of Lemma \ref{lem.P0P1}. We readily infer the behavior given in \eqref{behP2}.
\end{proof}
Skipping for the moment the local analysis of the system \eqref{PSsyst2} near the remaining finite critical point $P_3$ (postponed to the next section), we are ready to study the critical points at space infinity of the system.

\medskip

\textbf{Local analysis of the critical points at infinity.} In order to study such points, we pass to the Poincar\'e hypersphere following the theory given for example in \cite[Section 3.10]{Pe}. We introduce the new variables $(\overline{X},\overline{Y},\overline{Z},W)$ such that
$$
X=\frac{\overline{X}}{W}, \ Y=\frac{\overline{Y}}{W}, \ Z=\frac{\overline{Z}}{W},
$$
and according to \cite[Theorem 4, Section 3.10]{Pe}, the critical points at space infinity of the phase space associated to the system \eqref{PSsyst2} lie on the equator of the Poincar\'e hypersphere, that is, at points of the form $(\overline{X},\overline{Y},\overline{Z},0)$, and where the following system is fulfilled:
\begin{equation}\label{Poincare1}
\left\{\begin{array}{ll}\overline{X}Q_2(\overline{X},\overline{Y},\overline{Z})-\overline{Y}P_2(\overline{X},\overline{Y},\overline{Z})=0,\\
\overline{X}R_2(\overline{X},\overline{Y},\overline{Z})-\overline{Z}P_2(\overline{X},\overline{Y},\overline{Z})=0,\\
\overline{Y}R_2(\overline{X},\overline{Y},\overline{Z})-\overline{Z}Q_2(\overline{X},\overline{Y},\overline{Z})=0,\end{array}\right.
\end{equation}
together with the obvious condition $\overline{X}^2+\overline{Y}^2+\overline{Z}^2=1$, where $P_2$, $Q_2$ and $R_2$ are the homogeneous second degree parts of the polynomials in the right hand side of the system \eqref{PSsyst2}, that is
\begin{equation*}
\begin{split}
&P_2(\overline{X},\overline{Y},\overline{Z})=\frac{m-1}{2}\overline{X}\overline{Y}-\overline{X}^2,\\
&Q_2(\overline{X},\overline{Y},\overline{Z})=-\frac{m+1}{2}\overline{Y}^2,\\
&R_2(\overline{X},\overline{Y},\overline{Z})=\overline{Z}((m-1)\overline{Y}+\sigma\overline{X}).
\end{split}
\end{equation*}
The system \eqref{Poincare1} thus becomes after straightforward calculations
\begin{equation}\label{Poincare2}
\left\{\begin{array}{ll}\overline{X}\overline{Y}(\overline{X}-m\overline{Y})=0,\\
\overline{X}\overline{Z}\left((\sigma+1)\overline{X}+\frac{m-1}{2}\overline{Y}\right)=0,\\
\overline{Y}\overline{Z}\left(\sigma\overline{X}+\frac{3m-1}{2}\overline{Y}\right)=0,\end{array}\right.
\end{equation}
and taking into account that we are only considering the quarter of the equator of the hypersphere where $\overline{X}\geq0$ and $\overline{Z}\geq0$, we find the following five critical points:
$$
Q_1=(1,0,0,0), \ \ Q_{2,3}=(0,\pm1,0,0), \ \ Q_4=(0,0,1,0), \ \
Q_5=\left(\frac{m}{\sqrt{1+m^2}},\frac{1}{\sqrt{1+m^2}},0,0\right).
$$
We analyze below the local behavior of the orbits in the phase space near each of these points, which is quite similar to the one in \cite[Section 2]{IS2}.
\begin{lemma}[Analysis of the point $Q_1$]\label{lem.Q1}
The critical point at infinity represented as $Q_1=(1,0,0,0)$ in the Poincar\'e hypersphere is an unstable node. The orbits going out of this point to the finite part of the phase space contain profiles $f(\xi)$ such that $f(0)=a>0$ with any possible behavior of the derivative $f'(0)$.
\end{lemma}
\begin{proof}
The proof is rather similar to the one of the analogous result \cite[Lemma 2.5]{IS2}. We infer from part (a) of \cite[Theorem 5, Section 3.10]{Pe} that the flow of the system near $Q_1$ is topologically equivalent to the flow near the origin of the phase space in the system
\begin{equation}\label{systinf1}
\left\{\begin{array}{ll}-\dot{y}=-y+my^2+zw-w^2,\\
-\dot{z}=-(\sigma+1)z-\frac{m-1}{2}yz,\\
-\dot{w}=-w+\frac{m-1}{2}yw,\end{array}\right.
\end{equation}
where the minus sign has been chosen in the system \eqref{systinf1} in order to match the direction of the flow. This is noticed, for example, from the first equation of the original system \eqref{PSsyst2},
$$
\dot{X}=\frac{1}{2}X[(m-1)Y-2X],
$$
which gives $\dot{X}<0$ in a neighborhood of $Q_1$, taking into account that $|X/Y|\to+\infty$ near this point. It is immediate to see that the origin is an unstable node in the equivalent system \eqref{systinf1}. In order to establish the behavior of the profiles contained in the orbits going out of $Q_1$, we notice that in \eqref{systinf1} we have
$$
\frac{dy}{dw}\sim\frac{y}{w},
$$
or equivalently $y\sim Cw$, that is $Y/X\sim C/X$ or equivalently $Y\sim C$, where $C\in\real$ is an arbitrary constant. We infer by integration that
$$
f(\xi)\sim(C\xi+\overline{C})^{2/(m-1)}, \ {\rm as} \ \xi\to0, \quad \overline{C}>0, \ C\in\real,
$$
and the conclusion.
\end{proof}
\begin{lemma}[Analysis of the points $Q_2$ and $Q_3$]\label{lem.Q23}
The critical points at infinity represented as $Q_{2,3}=(0,\pm1,0,0)$ in the Poincar\'e hypersphere are an unstable node, respectively a stable node. The orbits going out of $Q_2$ to the finite part of the phase space contain profiles $f(\xi)$ such that there exists $\xi_0\in(0,\infty)$ with $f(\xi_0)=0$,
$f'(\xi_0)=+\infty$. The orbits entering the point $Q_3$ and coming from the finite part of the phase space contain profiles $f(\xi)$ such that there exists $\xi_0\in(0,\infty)$ with $f(\xi_0)=0$, $f'(\xi_0)=-\infty$.
\end{lemma}
\begin{proof}
We infer from part (b) of \cite[Theorem 5, Section 3.10]{Pe} that the flow of the system near the points $Q_2$ and $Q_3$ is topologically equivalent to the flow near the origin of the phase space in the system
\begin{equation}\label{systinf2}
\left\{\begin{array}{ll}\pm\dot{x}=-mx+x^2+xw^2-zw^2,\\
\pm\dot{z}=-\frac{3m-1}{2}z-\sigma xz-z^2+zw^2,\\
\pm\dot{w}=-\frac{m+1}{2}w-zw+w^3,\end{array}\right.
\end{equation}
where we have to choose the minus sign for one of the points and the plus sign for the other point. Similarly as in \cite[Lemma 2.6]{IS2}, we deduce from the second equation of the original system \eqref{PSsyst2}, that is
$$
\dot{Y}=-\frac{m+1}{2}Y^2+1-Z,
$$
that $\dot{Y}<0$ in a neighborhood of both points $Q_2$ and $Q_3$, which gives the direction of the flow (from right to left) and shows that we have to choose the minus sign in the system \eqref{systinf2} near $Q_2$ and the plus sign near $Q_3$. It follows that $Q_2$ is an unstable node and $Q_3$ is a stable node. To establish the local behavior, we proceed as in the end of the proof of Lemma 2.6 in our previous work \cite{IS2}, by noticing that
$$
\frac{dx}{dw}\sim\frac{2m}{m+1}\frac{x}{w},
$$
in a neighborhood of any of the two points, which implies that $x\sim C|w|^{2m/(m+1)}$ or equivalently, in the original variables,
$$
X\sim C|Y|^{-(m-1)/(m+1)}.
$$
Using \eqref{PSchange} and noticing that on the orbits going out of $Q_2$ (respectively entering $Q_3$) we have $Y\to\infty$ for $Q_2$ (respectively $Y\to-\infty$ for $Q_3$) and $X\to0$ for both points, we obtain by integration that
\begin{equation}\label{interm1}
f(\xi)\sim\left(C_1\xi^{2m/(m-1)}+C_2\right)^{1/m},
\end{equation}
and also that \eqref{interm1} occurs for $\xi\to\xi_0\in(0,\infty)$, since we easily discard the possibilities $\xi\to\infty$ and $\xi\to0$ with an analysis similar to the one performed in the proof of Lemma \ref{lem.P0P1}. We notice that for the orbits entering $Q_3$, $Y<0$ in a neighborhood of the point $Q_3$, which means $f'(\xi)<0$ and thus $C_1<0$ and $C_2>0$ in \eqref{interm1}. This shows that the profiles contained in orbits entering $Q_3$ have a change of sign at some point $\xi_0\in(0,\infty)$ with $f(\xi_0)=0$ and $f'(\xi_0)=-\infty$. On the contrary, for the orbits going out of $Q_2$, $Y>0$ in a neighborhood of $Q_2$, thus $f'(\xi)>0$ and $C_1>0$ and one has to choose $C_2<0$ in \eqref{interm1} to obtain profiles with a change of sign at some point $\xi_0\in(0,\infty)$ with $f(\xi_0)=0$ and $f'(\xi_0)=+\infty$, as stated.
\end{proof}
For the critical point $Q_4$, which is non-hyperbolic, we do not have to perform the local analysis near it using a dynamical system approach. We have
\begin{lemma}[No orbits connecting to $Q_4$]\label{lem.Q4}
There are no solutions to Eq. \eqref{SSODE} such that
$$
\lim\limits_{\xi\to\infty}\xi^{\sigma}f(\xi)^{m-1}=+\infty.
$$
In particular, there are no orbits entering the critical point $Q_4$ from the finite part of the phase space associated to the system \eqref{PSsyst2}.
\end{lemma}
\begin{proof}
The first statement has been proved as \cite[Lemma 2.8]{IS2} in a more general case (including our case $p=m$). We thus deduce that there are no orbits in the finite part of the phase space such that $\lim\limits_{\xi\to\infty}Z(\xi)=+\infty$, consequently no orbits entering $Q_4$ coming from the interior of the Poincar\'e hypersphere, as stated.
\end{proof}
\begin{lemma}[Analysis of the point $Q_5$]\label{lem.Q5}
The critical point at infinity represented as $Q_5$ in the Poincar\'e hypersphere has a two-dimensional unstable manifold and a one-dimensional stable manifold. The orbits going out from this point into the finite region of the phase space contain profiles satisfying $f(0)=0$ and $f(\xi)\sim K\xi^{1/m}$ as $\xi\to0$ in a right-neighborhood of $\xi=0$.
\end{lemma}
\begin{proof}
The flow of the system in a neighborhood of $Q_5$ is again topologically equivalent to the flow of the system \eqref{systinf1} but in a neighborhood of the critical point $(y,z,w)=(1/m,0,0)$ in the notation of \eqref{systinf1}. Since $X\sim mY$ in a neighborhood of the point $Q_5$, we infer that
$$
\dot{X}=\frac{1}{2}X[(m-1)Y-2X]\sim-\frac{m(m+1)}{2}Y^2<0,
$$
thus we have to choose again the minus sign in front of the derivatives in the system \eqref{systinf1}. The linearization of the system \eqref{systinf1} near $(1/m,0,0)$ has the matrix
$$
M(Q_5)=\left(
         \begin{array}{ccc}
           -1 & 0 & 0 \\
           0 & \frac{2m(\sigma+1)+m-1}{2m} & 0 \\
           0 & 0 & \frac{m+1}{2m} \\
         \end{array}
       \right),
$$
with two positive eigenvalues and a negative one. It is obvious that the orbits going out of $Q_5$ on the two-dimensional unstable manifold enter the finite part of the phase space. In order to establish its local behavior near $Q_5$, we start from the fact that $X\sim mY$ in a neighborhood of $Q_5$, that is
$$
\xi^{-1}f(\xi)^{(m-1)/2}\sim\frac{2m}{m-1}\left(f^{(m-1)/2}\right)'(\xi), \quad {\rm as} \ \xi\to0,
$$
and after integration
$$
f^{(m-1)/2}(\xi)\sim C\xi^{(m-1)/2m}, \quad {\rm as} \ \xi\to0,
$$
where $C>0$ is an arbitrary positive constant, and the conclusion follows.
\end{proof}

\section{No positive blow up profiles. Proof of Theorem \ref{th.nontail}}\label{sec.bifur}

In this section we perform the local analysis of the phase space associated to the system \eqref{PSsyst2} in a neighborhood of the critical point $P_3=(0,0,1)$. Let us notice first that the linearization of the system in a neighborhood of $P_3$ has the matrix
$$
M(P_3)=\left(
         \begin{array}{ccc}
           0 & 0 & 0 \\
           0 & 0 & -1 \\
           \sigma & m-1 & 0 \\
         \end{array}
       \right),
$$
with three eigenvalues with zero real part: $\lambda_1=0$, $\lambda_2=i\sqrt{m-1}$, $\lambda_3=-i\sqrt{m-1}$. Thus, this point is a non-hyperbolic critical point and in order to analyze the system in a neighborhood of it we have to use more involved techniques specific to the bifurcation theory, more precisely by deducing the \emph{normal form} of the system in a neighborhood of the point using a result which is typical for the \emph{Hopf-fold bifurcations}. We prove the following
\begin{lemma}\label{lem.P3}
There is no orbit entering the critical point $P_3$ coming from the region $\{X>0\}$ of the phase space.
\end{lemma}
\begin{proof}
The proof is rather technical and involved and will be divided into several steps.

\medskip

\noindent \textbf{Step 1.} Following the recipe in \cite[Section 3.1F, p.331]{Wig}, we perform a change of variable by letting
\begin{equation}\label{interm2}
v=(m-1)Y+\sigma X, \quad u=\sqrt{m-1}(Z-1), \quad z=X
\end{equation}
and after straightforward calculations, the system \eqref{PSsyst2} is transformed into the new system
\begin{equation}\label{interm3}
\left\{\begin{array}{ll}\dot{v}=-\sqrt{m-1}u+\frac{(3m+1)\sigma}{2(m-1)}zv-\frac{\sigma(m\sigma+m-1)}{m-1}z^2-\frac{m+1}{2(m-1)}v^2,\\
\dot{u}=\sqrt{m-1}v+uv,\\
\dot{z}=-\frac{\sigma+2}{2}z^2+\frac{1}{2}zv.\end{array}\right.
\end{equation}
We next follow the recipe given in \cite[Section 8.5]{KuBook} in order to establish the first terms of the normal form of the system \eqref{interm3}. To this end, we perform one more change of variable by letting $w=v+iu$, or equivalently
$$
v=\frac{w+\overline{w}}{2}, \quad u=\frac{w-\overline{w}}{2i}.
$$
We calculate $\dot{w}=\dot{v}+i\dot{u}$ and, using the equations in \eqref{interm3} and transforming $(v,u)$ into $(w,\overline{w})$, we find
\begin{equation}\label{interm4}
\begin{split}
\dot{w}&=i\sqrt{m-1}w-\frac{\sigma(m\sigma+m-1)}{m-1}z^2+\left(\frac{1}{4}-\frac{m+1}{8(m-1)}\right)w^2\\
&-\left(\frac{1}{4}+\frac{m+1}{8(m-1)}\right)\overline{w}^2+\frac{(3m+1)\sigma}{4(m-1)}(zw+z\overline{w})-\frac{m+1}{4(m-1)}w\overline{w},
\end{split}
\end{equation}
and
\begin{equation}\label{interm5}
\dot{z}=-\frac{\sigma+2}{2}z^2+\frac{1}{4}zw+\frac{1}{4}z\overline{w}.
\end{equation}

\medskip

\noindent \textbf{Step 2. Obtaining the normal form.} We already put the system \eqref{PSsyst2} in the form used in \cite[Section 8.5]{KuBook}, that is the system formed by the equations \eqref{interm4} and \eqref{interm5}, whose nonlinear parts are denoted $g(z,w,\overline{w})$ (for the equation \eqref{interm5}), respectively $h(z,w,\overline{w})$ (for the equation \eqref{interm4}). The Taylor expansions of $g$ and $h$ have, accordingly with the notation on \cite[p. 332-333]{KuBook}, the forms
$$
g(z,w,\overline{w})=\sum\limits_{j+k+l\geq2}\frac{1}{j!k!l!}g_{jkl}z^jw^k\overline{w}^l, \quad h(z,w,\overline{w})=\sum\limits_{j+k+l\geq2}\frac{1}{j!k!l!}h_{jkl}z^jw^k\overline{w}^l,
$$
whose coefficients in our case read
\begin{equation}\label{interm6}
g_{200}=-(\sigma+2), \ g_{110}=g_{101}=\frac{1}{4}, \ g_{020}=g_{002}=g_{011}=0,
\end{equation}
and (using the notation with $h$ for the $w$ equation)
\begin{equation}\label{interm7}
\begin{split}
&h_{200}=-\frac{2\sigma(m\sigma+m-1)}{m-1}, \ h_{020}=\frac{1}{2}-\frac{m+1}{4(m-1)}, \ h_{002}=-\left(\frac{1}{2}+\frac{m+1}{4(m-1)}\right),\\
&h_{110}=h_{101}=\frac{(3m+1)\sigma}{4(m-1)}, \ h_{011}=-\frac{m+1}{4(m-1)}.
\end{split}
\end{equation}
We now use \cite[Lemma 8.9]{KuBook} in order to obtain the coefficients of the Poincar\'e normal form of the system starting from the coefficients in \eqref{interm6} and \eqref{interm7}. More precisely, the Poincar\'e normal form of the system writes (adapting \cite[Lemma 8.9]{KuBook} to our case and notation)
\begin{equation}\label{normal}
\left\{\begin{array}{ll}\dot{z}=\frac{1}{2}G_{200}z^2+G_{011}|w|^2+\frac{1}{6}G_{300}z^3+G_{111}z|w|^2+O(|(z,w,\overline{w})|^4),\\
\dot{w}=i\sqrt{m-1}w+H_{110}zw+\frac{1}{2}H_{210}z^2w+\frac{1}{2}H_{021}w|w|^2+O(|(z,w,\overline{w})|^4),\end{array}\right.
\end{equation}
where the coefficients of the normal form \eqref{normal} are
\begin{equation}\label{interm8}
G_{200}=g_{200}=-(\sigma+2), \ G_{011}=G_{111}=G_{300}=0,
\end{equation}
and
\begin{equation}\label{interm9}
H_{110}=h_{110}=\frac{(3m+1)\sigma}{4(m-1)},
\end{equation}
\begin{equation}\label{interm10}
\begin{split}
H_{210}&=\frac{i}{2\sqrt{m-1}}\left[h_{200}(h_{020}-2g_{110})-|h_{101}|^2-h_{011}\overline{h}_{200}\right]\\
&=-\frac{i}{2\sqrt{m-1}}\frac{(3m+1)^2\sigma^2}{16(m-1)^2},\\
H_{021}&=\frac{i}{2\sqrt{m-1}}\left[h_{011}h_{020}-\frac{1}{2}g_{020}h_{101}-2|h_{011}|^2-\frac{1}{3}|h_{002}|^2\right]\\
&=\frac{i}{2\sqrt{m-1}}\left[-\frac{(m+1)^2}{12(m-1)^2}-\frac{5(m+1)}{24(m-1)}-\frac{1}{12}\right].
\end{split}
\end{equation}
Gathering all these calculations, we can thus write the obtained Poincar\'e normal form of the system:
\begin{equation*}
\left\{\begin{array}{ll}\dot{z}=-\frac{\sigma+2}{2}z^2+O(|(z,w,\overline{w})|^4),\\
\dot{w}=i\sqrt{m-1}w+\frac{(3m+1)\sigma}{4(m-1)}zw+\frac{1}{2}H_{210}z^2w+\frac{1}{2}H_{021}w|w|^2+O(|(z,w,\overline{w})|^4),\end{array}\right.
\end{equation*}
with $H_{210}$, $H_{021}$ given in \eqref{interm10}. Undoing the change of variable $w=v+iu$ in order to get back to the variables $(z,v,u)$ and keeping only the terms up to order two, we finally find the Poincar\'e normal form of the system \eqref{interm3}:
\begin{equation}\label{normal.syst}
\left\{\begin{array}{ll}\dot{z}=-\frac{\sigma+2}{2}z^2+O(|(z,v,u)|^3),\\
\dot{v}=-\sqrt{m-1}u+\frac{(3m+1)\sigma}{4(m-1)}zv+O(|(z,v,u)|^3),\\
\dot{u}=\sqrt{m-1}v+\frac{(3m+1)\sigma}{4(m-1)}zu+O(|(z,v,u)|^3).\end{array}\right.
\end{equation}

\medskip

\noindent \textbf{Step 3. Normal form in cylindrical coordinates.} Following again \cite[Section 3.1F]{Wig}, we transform the normal form \eqref{normal.syst} to cylindrical coordinates by letting $v=r\cos\,\theta$, $u=r\sin\,\theta$ and unchanged $z$. After standard calculations, we obtain the following normal form in cylindrical coordinates:
\begin{equation}\label{normal.cyl}
\left\{\begin{array}{ll}\dot{z}=-\frac{\sigma+2}{2}z^2+O(|(z,r)|^3),\\
\dot{r}=\frac{(3m+1)\sigma}{4(m-1)}zr+O(|(z,r)|^3),\\
\dot{\theta}=\sqrt{m-1}+O(|(z,r)|^3),\end{array}\right.
\end{equation}
and we notice that $\dot{r}>0$ whenever the connection does not lie in the plane $r=0$ or $z=0$. If we are not in such case, we can integrate (up to third order) the system formed by the first two equations in \eqref{normal.cyl} to get in a neighborhood of the origin in the system \eqref{normal.cyl}
$$
\frac{dz}{dr}\sim-\frac{(\sigma+2)(m-1)}{2(3m+1)\sigma}\frac{z}{r},
$$
or equivalently
$$
z\sim Cr^{-(\sigma+2)(m-1)/2\sigma(3m+1)},
$$
thus the trajectories in a neighborhood of the critical point (that became the origin in the system \eqref{normal.cyl}) tend to hyperbolas that do not enter the origin (which behaves as a repeller in our case). The only orbits that may connect to the origin in the system \eqref{normal.cyl} are those contained either in the plane $\{r=0\}$ or in the plane $\{z=0\}$. Taking into account that in the original variables $z=X$, the orbits contained in the plane $\{z=0\}$ are in fact contained in the plane $\{X=0\}$ and they do not contain profiles. As for the orbits contained in the plane $\{r=0\}$, this means at the same time that $v=0$ and $u=0$, and we infer from \eqref{interm2} that these orbits are contained in the line of equations $Z=1$, $(m-1)Y+\sigma X=0$. But the latter line is nothing else than the hyperbola given by
\begin{equation}\label{hyp}
f(\xi)=\left(\frac{1}{m-1}\right)^{1/(m-1)}\xi^{-\sigma/(m-1)},
\end{equation}
which is not a profile for $\sigma>0$ (although it is the constant solution in the non-weighted case $\sigma=0$). We thus conclude that there is no blow up profile entering the critical point $P_3$.
\end{proof}
The numerical experiment represented in Figure \ref{fig3} shows how an orbit in the phase space, passing through a point which lies very close to the critical point $P_3$ (namely $(X,Y,Z)=(0.05,0.01,1.01)$), does not enter the point $P_3$, as proved above. In fact, the orbit describes outgoing spirals around the point $P_3$ before going out of the spiral and approaching the critical point $Q_3$.

\begin{figure}[ht!]
  \begin{center}
  \subfigure[Plot in 3D]{\includegraphics[width=7.5cm,height=6cm]{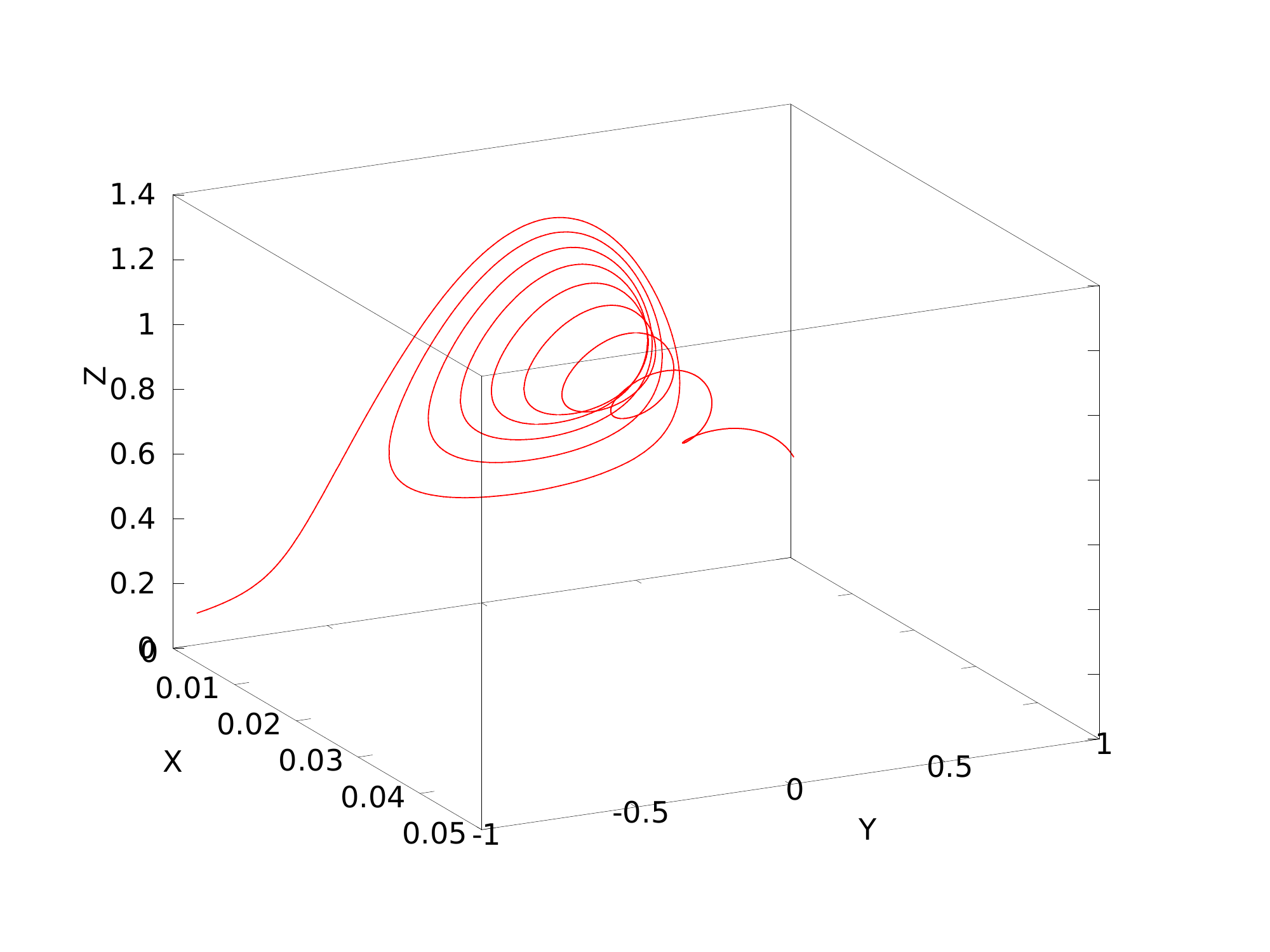}}
  \subfigure[Plot in 2D]{\includegraphics[width=7.5cm,height=6cm]{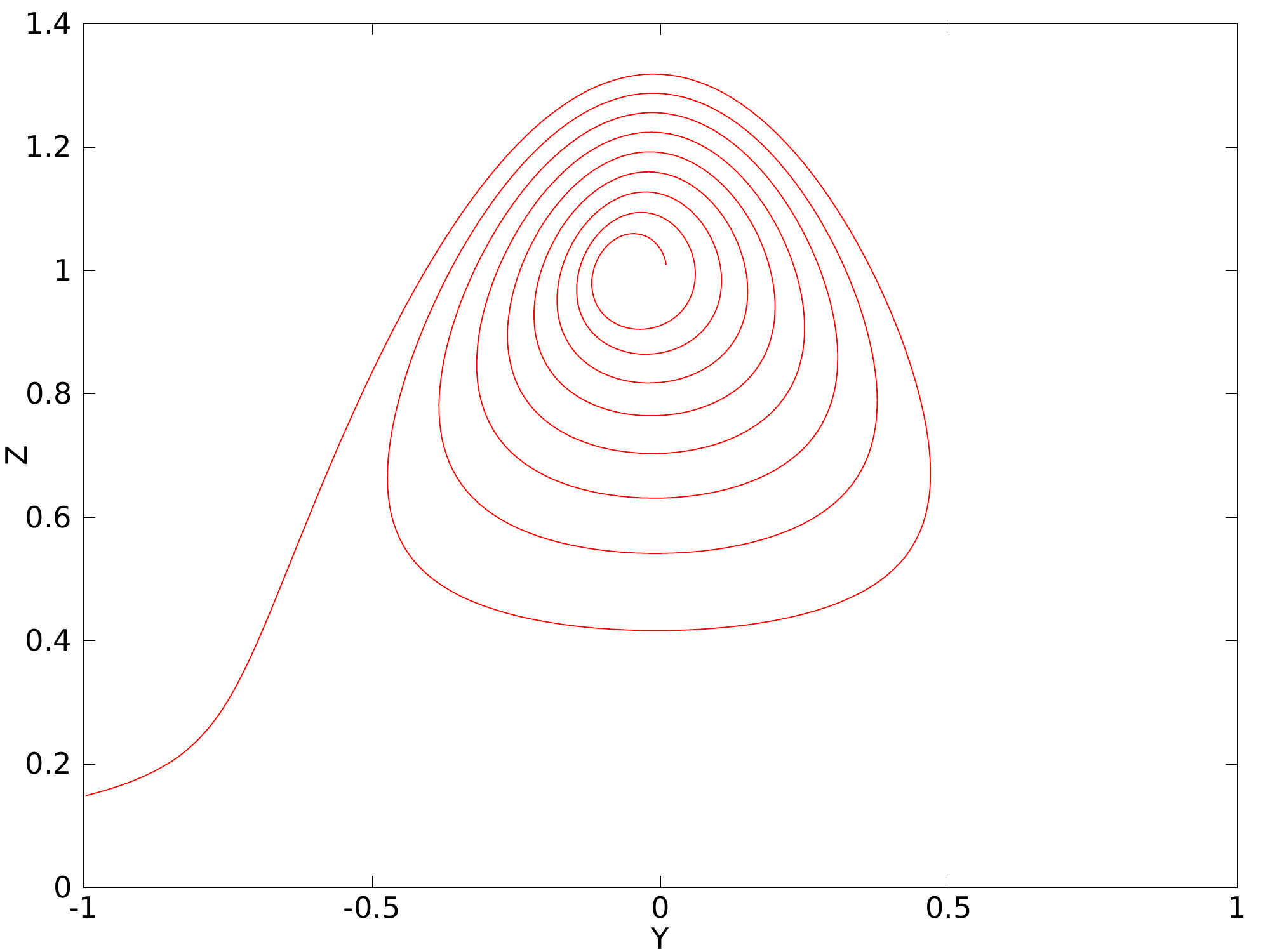}}
  \end{center}
  \caption{Orbit passing through a point very close to $P_3$ and going out. Experiment for $m=2$, $p=2$ and $\sigma=1$ with plot in 3D and projection on the $(Y,Z)$ plane.}\label{fig3}
\end{figure}

After this rather lengthy and technical lemma, the proof of Theorem \ref{th.nontail} becomes a corollary.
\begin{proof}[Proof of Theorem \ref{th.nontail}] Assume for contradiction that there exists at least a profile $f(\xi)$ solution to \eqref{SSODE}, such that $f(\xi)>0$ for any $\xi\geq0$. Then this profile has to be contained in an orbit in the phase space associated to the system \eqref{PSsyst2}, and from the list of local behaviors near the critical points, the only behavior which is positive as $\xi\to\infty$ would be the one given by entering the point $P_3$. But we just proved in Lemma \ref{lem.P3} that there is no orbit entering $P_3$ that contains blow up profiles. Thus, there are no profiles that are positive everywhere, as claimed.
\end{proof}

\noindent \textbf{Remark.} In the homogeneous case $\sigma=0$, there is one profile that is positive everywhere and enters the point $P_3$, which is the constant profile
$$
f(\xi)=\left(\frac{1}{m-1}\right)^{1/(m-1)}.
$$
When passing to $\sigma>0$, this profile does no longer exist, and an equivalent of it is the hyperbola \eqref{hyp}, which is not a solution. This is an interesting difference introduced by the weight.

\section{Global analysis in the phase space. Existence and multiplicity for $\sigma>0$ small}\label{sec.exist}

We devote this section to establishing the connections in the phase space between the critical points, emphasizing in particular on those orbits which contain profiles with the desired behavior. As we shall see, a cylinder splitting the space into two separate regions will be of utmost importance. At the end, we will prove the existence and multiplicity result Theorem \ref{th.exist} for sufficiently small $\sigma>0$.

\subsection{The cylinder}\label{subsec.global}

We begin with the remark that on the invariant plane $\{X=0\}$, the system \eqref{PSsyst2} reduces to
\begin{equation}\label{syst.X0}
\left\{\begin{array}{ll}\dot{Y}=-\frac{m+1}{2}Y^2+1-Z,\\ \dot{Z}=(m-1)YZ,\end{array}\right.
\end{equation}
which is nothing else that the same system obtained when letting $\sigma=0$ and considering only the last two equations. This system leads to
$$
\frac{dY}{dZ}=\frac{-(m+1)Y^2/2+1-Z}{(m-1)YZ},
$$
which can be integrated explicitly to find its general solution
\begin{equation}\label{interm11}
Y^2=KZ^{-(m+1)/(m-1)}-\frac{(m+1)Z-2m}{m(m+1)}, \quad K\in\real.
\end{equation}
The most relevant solution in this one-parameter family is the one with $K=0$, namely
\begin{equation}\label{cylinder}
Y^2=\frac{2}{m+1}-\frac{1}{m}Z,
\end{equation}
which is a separatrix of the phase plane in the non-weighted case $\sigma=0$. For $K>0$, the corresponding curve given by \eqref{interm11} lies in the region of the plane $(Y,Z)$ outside the explicit curve given in \eqref{cylinder}, while for $K<0$ it lies in the bounded region inside the same curve. This suggests us to consider, in the whole phase space, the cylinder of equation \eqref{cylinder} with $X\geq0$. The normal direction to the cylinder at some point is given by the vector $(0,2Y,1/m)$, thus the direction of the flow on the surface of the cylinder is given by the sign of the expression
$$
F(X,Y,Z):=-(m+1)Y^3+2Y-2YZ+\frac{m-1}{m}YZ+\frac{\sigma}{m}XZ.
$$
Taking into account that $Z=2m/(m+1)-mY^2$, we obtain that
$$
F(X,Y,Z)=\frac{\sigma}{m}XZ\geq0,
$$
thus the flow on the cylinder has always one direction, towards the exterior of it. This implies in particular the following very useful result:
\begin{lemma}\label{lem.cyl}
Any orbit that contains a point outside the cylinder of equation \eqref{cylinder}, remains outside the cylinder forever after this point.
\end{lemma}
We also notice that the three critical points $P_0$, $P_1$ and $P_2$ lie on the cylinder, thus it is important to establish whether the orbits going out of $P_0$ and $P_2$, respectively entering $P_1$, do this through the region outside or inside the cylinder. This is the goal of the following two lemmas.
\begin{lemma}\label{lem.P0P2global}
The orbits going out of both critical points $P_0$ and $P_2$, do this in the region that lie in the exterior of the cylinder of equation \eqref{cylinder}.
\end{lemma}
\begin{proof}
Let us notice first that on the cylinder we have $\dot{Y}\leq0$. Indeed, we have $Z=2m/(m+1)-mY^2$, thus
$$
\dot{Y}=1-Z-\frac{m+1}{2}Y^2=-\frac{m-1}{2m}\left(\frac{2}{m+1}-Y^2\right)\leq0,
$$
since $|Y|\leq h_0=\sqrt{2/(m+1)}$. Assume now for contradiction that there is at least one orbit going out of the critical point $P_0$ inside the cylinder of equation \eqref{cylinder}. We then compare for the same value of $Y\in(0,h_0)$ (and sufficiently close to $h_0$), the coordinates $Z_2$ (on the cylinder) and $Z_1$ (the true value of $Z$ for the given $Y$ on the orbit that we assumed to go inside the cylinder). Thus, $Z_2>Z_1$ by the assumption that the orbit goes inside the cylinder. On the other hand, $Z_2=Z_2(Y)$ satisfies the equation of the cylinder, namely
$$
\frac{dZ_2}{dY}=\frac{(m-1)YZ_2}{-(m+1)Y^2/2+1-Z_2},
$$
while $Z_1=Z_1(Y)$ is a true orbit in the phase space, thus a solution to the whole system \eqref{PSsyst2}, that is
$$
\frac{dZ_1}{dY}=\frac{(m-1)YZ_1+\sigma XZ_1}{-(m+1)Y^2/2+1-Z_1}.
$$
Hence
\begin{equation}\label{interm13}
\begin{split}
\frac{d(Z_2-Z_1)}{dY}&=\frac{(m-1)Y(Z_2-Z_1)(1-(m+1)Y^2/2)}{(-(m+1)Y^2/2+1-Z_1)(-(m+1)Y^2/2+1-Z_2)}\\
&-\frac{\sigma XZ_1}{-(m+1)Y^2/2+1-Z_1}.
\end{split}
\end{equation}
Since on the one hand $Z_2=Z_2(Y)$ lies on the cylinder (where $\dot{Y}<0$), and on the other hand in a neighborhood of $P_0$, $Z_1=Z_1(Y)$ lies on an orbit which is decreasing with respect to the variable $Y$ (as it goes towards the interior of the cylinder), we infer that
$$
-\frac{m+1}{2}Y^2+1-Z_2<0, \quad -\frac{m+1}{2}Y^2+1-Z_1<0
$$
and we deduce from \eqref{interm13} that $d(Z_2-Z_1)/dY>0$ for any $\sigma>0$ in a sufficiently small left-neighborhood of $P_0$. Thus, there exists some $\epsilon>0$ sufficiently small such that, for $Y\in(h_0-\epsilon,h_0)$, the distance $Z_2(Y)-Z_1(Y)$ increases with $Y$. But this is a contradiction with the fact that for $Y=h_0$, we have $Z_1(h_0)=Z_2(h_0)=0$, as both curves (the true orbit described by $Z_1$ and the cylinder described by $Z_2$) are assumed to start from $P_0$. This contradiction shows that all the orbits going out of $P_0$ do that strictly in the exterior of the cylinder of equation \eqref{cylinder} (since the cylinder itself is not a solution for $\sigma>0$).

It remains to study how the unique orbit (according to Lemma \ref{lem.P2}) going out of $P_2$ behaves with respect to the cylinder. This is much easier than the previous analysis, as this unique orbit goes out tangent to the direction given by the eigenvector
$$
e_3=\left(-\frac{m-1}{2(\sigma+3)},-1,\frac{\sigma(m-1)+4m}{2}h_0\right)
$$
corresponding to the only positive eigenvalue $\lambda_3=(m-1)(\sigma+2)h_0/2$. The normal to the cylinder at the point $P_2$ has the direction
$$
n(P_2)=\left(0,2h_0,\frac{1}{m}\right).
$$
Since
$$
n(P_2)\cdot e_3=\frac{\sigma(m-1)}{2m}h_0>0,
$$
it follows that the orbit going out of $P_2$ does this in the exterior region to cylinder of equation \eqref{cylinder}, as claimed.
\end{proof}
\begin{lemma}\label{lem.P1global}
The orbits entering the critical point $P_1$ in the phase space do this in the region that lie in the interior of the cylinder of equation \eqref{cylinder}.
\end{lemma}
\begin{proof}
The proof is very similar to the one we did for the orbits going out of $P_0$. Assume for contradiction that there is at least one orbit in the phase space associated to the system \eqref{PSsyst2} entering $P_1$ outside the cylinder. Define again for the same $Y\in(-h_0,0)$, the corresponding coordinates $Z_1=Z_1(Y)$ on this orbit and $Z_2=Z_2(Y)$ on the cylinder. From the assumption that the orbit enters $P_1$ outside the cylinder, we infer that $Z_2(Y)<Z_1(Y)$ for $Y\in(-h_0,0)$ sufficiently close to $-h_0$. It also follows easily that $\dot{Y}<0$ both on the cylinder (already proved in the proof of Lemma \ref{lem.P0P2global}) and on the orbit that lies outside of it, thus once more
$$
-\frac{m+1}{2}Y^2+1-Z_2<0, \quad -\frac{m+1}{2}Y^2+1-Z_1<0.
$$
We proceed as in the proof of Lemma \ref{lem.P0P2global} and compute for comparison
\begin{equation}\label{interm14}
\begin{split}
\frac{d(Z_1-Z_2)}{dY}&=\frac{(m-1)Y(Z_1-Z_2)(1-(m+1)Y^2/2)}{(-(m+1)Y^2/2+1-Z_1)(-(m+1)Y^2/2+1-Z_2)}\\
&+\frac{\sigma XZ_1}{-(m+1)Y^2/2+1-Z_1}<0,
\end{split}
\end{equation}
for any $\sigma>0$ and $-h_0<Y<0$, since both terms in the right hand side of \eqref{interm14} are negative. Thus, there exists some $\epsilon>0$ sufficiently small such that, for $Y\in(-h_0,-h_0+\epsilon)$, the distance $Z_1(Y)-Z_2(Y)$ increases while $Y$ decreases, thus becoming even more positive. But this is a contradiction with the fact that for $Y=-h_0$, we have $Z_1(-h_0)=Z_2(-h_0)=0$, as both curves are assumed to enter $P_1$. This contradiction shows that all the orbits entering $P_1$ do that strictly through the interior of the cylinder of equation \eqref{cylinder} (since the cylinder itself is not a solution for $\sigma>0$).
\end{proof}
We picture in Figure \ref{fig5} the cylinder \eqref{cylinder} together with a sample orbit in the phase space from all the critical points discussed above: either going out of $P_0$ or $P_2$ (and in each case, we see how they travel outside the cylinder and overpass it) or entering $P_1$ (where we see how the orbit stays in the interior of the cylinder, oscillates for a few times and then starts to increase rapidly with respect to the $X$ coordinate).
\begin{figure}[ht!]
  \begin{center}
  \includegraphics[width=12cm,height=10cm]{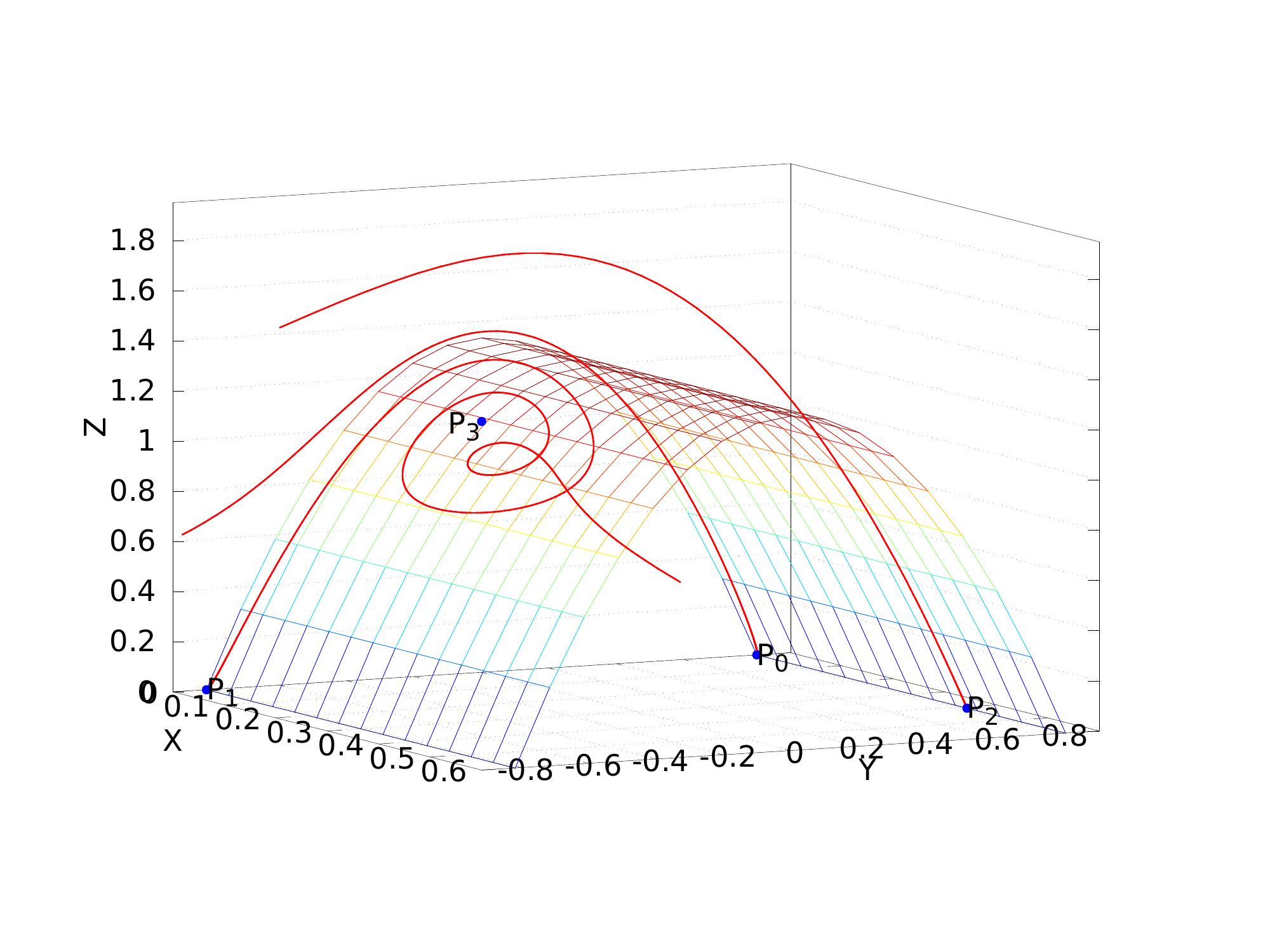}
  \end{center}
  \caption{The cylinder and sample orbits going out of $P_0$ and $P_2$, respectively entering $P_1$.}\label{fig5}
\end{figure}

All these results gather into a consequence of utmost importance for the rest of our analysis, that we state below.
\begin{proposition}\label{prop.interf}
Any profile $f(\xi)$ solution to \eqref{SSODE} for $\xi\in(0,\infty)$ having an interface at some finite point $\xi_0\in(0,\infty)$, is positive at the origin, that is $f(0)=a>0$.
\end{proposition}
\begin{proof}
We infer from the local analysis performed in Section \ref{sec.local} and specially Lemma \ref{lem.P0P1} that profiles with interface belong to orbits entering the critical point $P_1$. Lemma \ref{lem.P1global} shows that such orbits enter $P_1$ through the interior of the cylinder of equation \eqref{cylinder}. On the other hand, we deduce from Lemmas \ref{lem.P0P2global} and \ref{lem.cyl} that the orbits going out of the critical points $P_0$ and $P_2$ (and similarly for the critical points $Q_2$ and $Q_5$ at infinity, which start with coordinate $Y\to\infty$) do that outside the cylinder and thus, such orbits stay forever outside the cylinder. By discarding all the other options, we find that all the orbits entering the critical point $P_1$ necessarily come from the unstable node $Q_1$ on the Poincar\'e hypersphere. Lemma \ref{lem.Q1} shows that profiles contained in such orbits satisfy $f(0)=a>0$ (with any possible slope at the origin), ending the proof.
\end{proof}
In order to establish existence or non-existence, it only remains to study whether such profiles contained in orbits connecting $Q_1$ and $P_1$ are good (that is, $f'(0)=0$) or not. This is the goal of the remaining part of the paper.

\subsection{Existence of multiple solutions for $\sigma>0$ small}\label{subsec.exist}

In this subsection we use, only for the proof of existence of multiple solutions, an alternative phase space associated to a system already used in our previous work \cite[Section 2]{IS2}. We thus let
\begin{equation}\label{PSchange2}
x(\eta)=f^{m-1}(\xi), \ y(\eta)=(f^{m-2}f')(\xi), \ z(\eta)=\xi, \ \frac{d\eta}{d\xi}=mx(\eta),
\end{equation}
and Eq. \eqref{SSODE} transforms into the following system
\begin{equation}\label{PSsyst1}
\left\{\begin{array}{ll}\dot{x}=m(m-1)xy,\\
\dot{y}=-my^2+\frac{1}{m-1}x-z^{\sigma}x^2,\\
\dot{z}=mx,\end{array}\right.
\end{equation}
where derivatives are taken with respect to the independent variable $\eta$. The interface point $P_1$ in the phase space associated to the system \eqref{PSsyst2} is mapped into the critical half-line $\{x=0, \ y=0, \ z>0\}$ in the system \eqref{PSsyst1}, and Proposition \ref{prop.interf} gives in the new system that any orbit entering a critical point $(0,0,\xi_0)$ of the half-line $\{x=0, \ y=0, \ z>0\}$ has to intersect the plane $\{z=0\}$ at some point $(x_0,y_0)$ with $x_0>0$ and $y_0\in\real$. Using this phase space, one can readily prove that for any $\xi_0\in(0,\infty)$ fixed, there exists a unique profile with interface at $\xi=\xi_0$. The proof is identical to that of \cite[Proposition 3.3]{IS2} and is left to the reader. We are now in a position to prove our existence theorem.
\begin{proof}[Proof of Theorem \ref{th.exist}]
We divide the proof into two steps.

\medskip

\noindent \textbf{Step 1. Existence of one solution for $\sigma>0$ small.} It is well known that for $\sigma=0$ there exists one explicit good profile with interface and whose maximum only depends on $m$ \cite{S4}
\begin{equation}\label{prof.homog}
F_0(\xi)=\left[\frac{2m}{(m+1)(m-1)}\right]^{1/(m-1)}\left(\cos^2\left(\frac{(m-1)\xi}{2\sqrt{m}}\right)\right)_{+}^{1/(m-1)},
\end{equation}
and due to the invariance to translations when $\sigma=0$, any translation in space of it is also a solution to \eqref{SSODE}. In particular, with slight translations in space of $F_0(\xi)$ in both directions, one can get a connection in the phase space associated to the system \eqref{PSsyst1} for $\sigma=0$ having an interface and cutting the plane $\{z=0\}$ in a point $(x_0,y_0)$ with $y_0>0$ and another connection cutting the plane $\{z=0\}$ in a point $(x_1,y_1)$ with $y_1<0$. Working by now with the first connection, let $\epsilon>0$ be so small such that the two-dimensional ball $B((x_0,y_0),\epsilon)$ is contained in the region $\{x>0, \ y>0\}$ of the plane $\{z=0\}$. Consider the paraboloid of equation
\begin{equation}\label{parab}
(x-x_0)^2+(y-y_0)^2+z=\epsilon^2, \qquad z\geq0,
\end{equation}
whose normal direction (pointing upwards) is given by the vector $(2(x-x_0),2(y-y_0),1)$. Thus, the direction of the flow on the paraboloid \eqref{parab} is given by the sign of the expression
$$
E(x,y,z)=2m(m-1)(x-x_0)xy+2(y-y_0)(-my^2+\frac{m-1}{2}x-z^{\sigma}x^2)+mx.
$$
We notice that $|x-x_0|\leq\epsilon$ and $|y-y_0|<\epsilon$ on the paraboloid \eqref{parab}, thus we readily deduce that
$$
E(x,y,z)>K\epsilon+m(x_0-\epsilon), \qquad K\in\real,
$$
which is positive for $\epsilon>0$ sufficiently small. This proves that any orbit crossing the paraboloid \eqref{parab} must intersect the plane $\{z=0\}$ in a point which lies in the interior of the paraboloid, that is, in the two-dimensional ball $B((x_0,y_0),\epsilon)$ and in particular with coordinate $y>0$. By continuity with respect to the parameter $\sigma$, we infer that there exists $\sigma_{+}>0$ such that for any $\sigma\in(0,\sigma_+)$ there exists an orbit corresponding to profiles with interface crossing the paraboloid \eqref{parab} (and thus intersecting the plane $\{z=0\}$ at points with coordinate $y>0$). Considering the orbit for $\sigma=0$ that intersects the plane $\{z=0\}$ at the point $(x_1,y_1)$ with $y_1<0$ and performing an analogous analysis as above, we find that there exists $\sigma_{-}>0$ such that for any $\sigma\in(0,\sigma_{-})$, there is an orbit corresponding to profiles with interface that touches the plane $\{z=0\}$ at points with coordinate $y<0$. Letting $\sigma_0=\min\{\sigma_{+},\sigma_{-}\}$, it follows from the continuity of the slopes that for any $\sigma\in(0,\sigma_0)$ there exists a good profile with interface solution to \eqref{SSODE}.

\medskip

\noindent \textbf{Step 2. Existence of multiple solutions.} The proof of this step follows similar ideas to the previous one. Let $k>0$ be given. The main point is that, instead of using as starting point for translations and continuity argument the explicit profile $F_0(\xi)$ given by \eqref{prof.homog}, we consider a kind of "generalized profile" (using an abuse of language) constructed by \emph{concatenating $k$ copies of $F_0(\xi)$ one after another} (that is, the forward interface point of the $i$-th copy is the backward interface point of the $i+1$-th copy). This new object has $k$ zeros in the region $\xi>0$ and by slightly translating it in space to the left and to the right, we find such "concatenated generalized profiles" with $k$ zeros but such that they touch the vertical axis with positive, respectively negative derivatives. Since Proposition \ref{prop.interf} gives that for $\sigma>0$, the profiles we shoot from the same interface points as the two translated "generalized profiles" with $k$ zeros have to touch directly the plane $\{z=0\}$ (without passing through another zero point), we can repeat the argument with small balls as in Step 1 to conclude that there exists $\sigma_{0,k}>0$ sufficiently small such that for any $\sigma\in(0,\sigma_{0,k})$, there exists a good profile with interface to \eqref{SSODE} having exactly $k$ local maxima and $k$ local minima. Thus, for any given positive integer $k$, letting
$$
\sigma_k:=\min\{\sigma_{0,j}:1\leq j\leq k\}>0,
$$
we infer that for any $\sigma\in(0,\sigma_k)$, there exist at least $k$ different good profiles with interface, each of them having a different number of local maxima.
\end{proof}
\begin{figure}[ht!]
  \begin{center}
  \includegraphics[width=10cm,height=7.5cm]{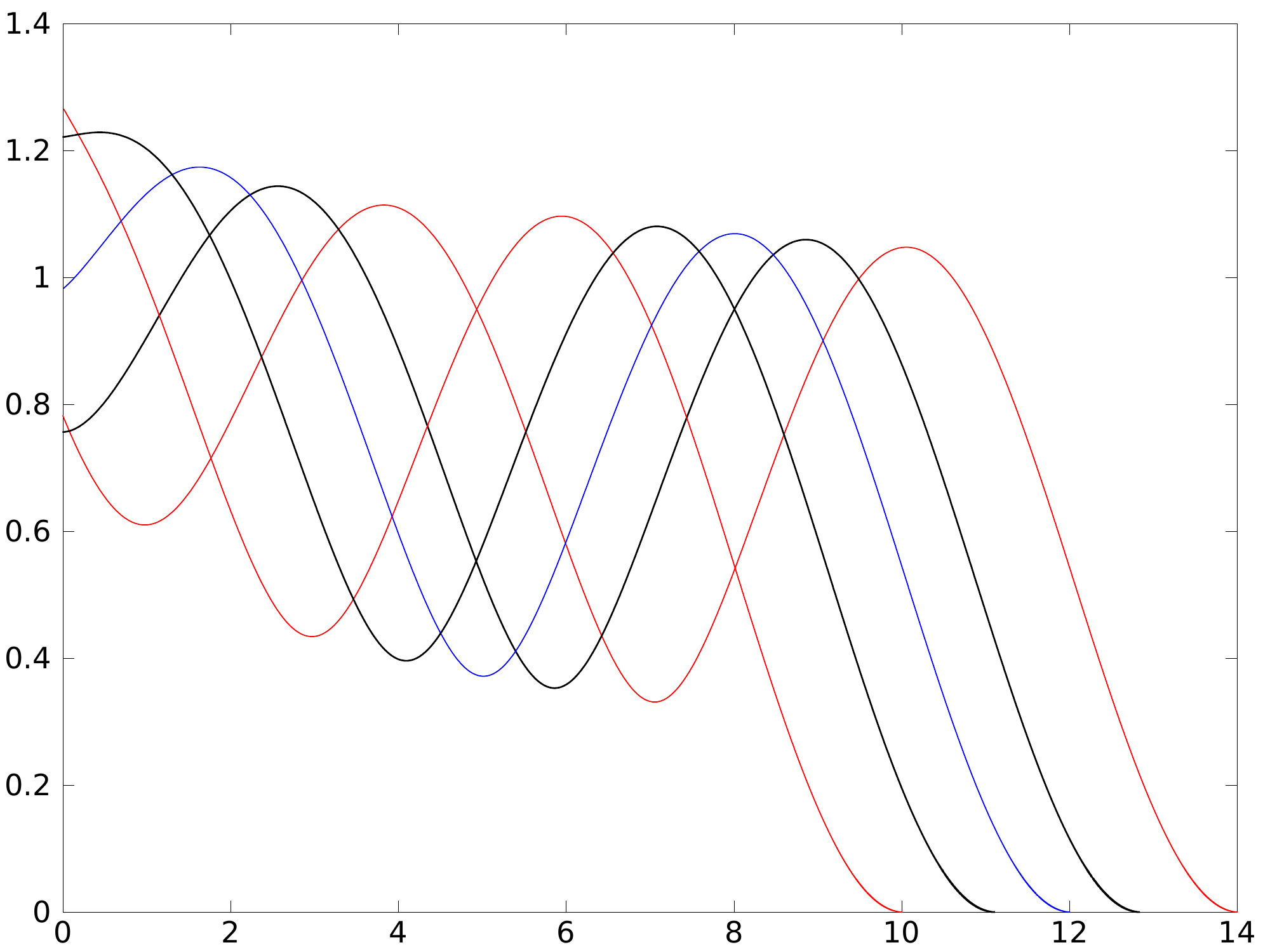}
  \end{center}
  \caption{Good and bad profiles with interface, with different slopes at $\xi=0$. Experiment for $m=2$, $p=2$ and $\sigma=0.1$}\label{fig4}
\end{figure}
Figure \ref{fig4} gives a visual representation of how the above proof works. We notice that the profiles with interface at $\xi_0=10$ and $\xi_0=14$ intersect the vertical axis with negative slope (that is, $f'(0)<0$ for such profiles), while the profile with interface at $\xi_0=12$ intersects the vertical axis with positive slope (that is, $f'(0)>0$ for this profiles). By continuity, between them there are two different good profiles with interface (whose interfaces in our experiment are approximately at $\xi_0=11.1$, respectively $\xi_0=12.83$). This is exactly the technique used in the proof of Theorem \ref{th.exist}.

\noindent \textbf{Remark.} Numerical evidences show that there should exist two different good profiles with interface for any fixed number of local maxima given, as for example the two good profiles with interface represented in Figure \ref{fig4}.

\section{Non-existence of good profiles with interface for large $\sigma$}\label{sec.nonexist}

As explained in the Introduction, for $\sigma>0$ large there is a striking difference: we pass from the existence of multiple solutions to the non-existence of any. This section is devoted to the proof of this result. The main idea in the proof is again geometric: we consider the hyperbola given by the graph of the function
\begin{equation}\label{hyp1}
f(\xi)=\left(\frac{1}{m(m-1)}\right)^{1/(m-1)}\xi^{-\sigma/(m-1)},
\end{equation}
and in a first step we shoot forward from the vertical axis with $f(0)=a>0$, $f'(0)=0$ to show that the first intersection point of the graph of the profile $f$ with the hyperbola \eqref{hyp1} is bounded from below. In a second step, we show that shooting backward from the interface point for $\sigma>0$ large enough, the last (in a backward sense) intersection point of the profiles with interfaces with the hyperbola \eqref{hyp1} goes below the bound found in the first step when shooting with profiles from the vertical axis. Thus the two shooting processes cannot meet and give rise to a full good profile with interface, provided $\sigma$ is sufficiently large.

\subsection{A monotonicity result}\label{subsec.monot}

In order to perform the program described in a few words above, we need a preliminary fact about the equation. Recalling the hyperbola \eqref{hyp}, it is immediate to deduce from \eqref{SSODE} that for any profile $f$ solution to \eqref{SSODE} for $\xi\in(0,\infty)$, any local maximum $\xi_0\in(0,\infty)$ of $f$ fulfills
$$
\xi_0^{\sigma}f^m(\xi_0)\geq\frac{1}{m-1}f(\xi_0),
$$
hence it lies above the hyperbola \eqref{hyp} and thus also above the hyperbola \eqref{hyp1}. We next show that the graphs of two profiles $f_1$, $f_2$ which are ordered at $\xi=0$ either with respect to the values of $f_{i}(0)$ or with respect to the values of $f_{i}'(0)$, remain ordered before crossing the hyperbola \eqref{hyp1} for the first time. More precisely
\begin{lemma}\label{lem.monot}
Let $\sigma>0$ and $f_1$, $f_2$ be two profiles such that one of the following two conditions is fulfilled:

(a) $f_1(0)=f_2(0)=a>0$ and $f_2'(0)>f_1'(0)=0$

(b) $f_2(0)>f_1(0)>0$ and $f_1'(0)=f_2'(0)=0$.

Then the graphs of $f_1$ and $f_2$ cannot intersect before they both have intersected the graph of the hyperbola \eqref{hyp1}.
\end{lemma}
\begin{proof}
We divide again the proof into several steps in order to ease its presentation.

\medskip

\noindent \textbf{Step 1. An integral identity for solutions.} Let $f(\xi)$ be a generic solution to \eqref{SSODE} and define $g(\xi)=f(\xi)^m$, which solves
\begin{equation}\label{interm17}
g''(\xi)=\frac{1}{m-1}g(\xi)^{1/m}-\xi^{\sigma}g(\xi).
\end{equation}
By multiplying with $g'(\xi)$ in \eqref{interm17} and then integrating on an interval $(0,\xi_0)$ for any $\xi_0>0$ fixed, we obtain
\begin{equation}\label{int.ident}
\begin{split}
(g')^{2}(\xi_0)&=(g')^2(0)+\frac{2m}{(m+1)(m-1)}\left[g(\xi_0)^{(m+1)/m}-g(0)^{(m+1)/m}\right]\\
&-\xi_0^{\sigma}g(\xi_0)^2+\sigma\int_0^{\xi_0}\xi^{\sigma-1}g(\xi)^2\,d\xi.
\end{split}
\end{equation}

\medskip

\noindent \textbf{Step 2. Monotonicity when condition (a) is fulfilled.} Defining $g_1=f_1^m$ and $g_2=f_2^m$, condition (a) implies that $g_1(0)=g_2(0)$ and $g_2'(0)>0=g_1'(0)$. We infer that in a right-neighborhood of $\xi=0$, $g_2(\xi)>g_1(\xi)$. Let $\xi_0>0$ be the first point of intersection of the graphs of $g_1$ and $g_2$, that is $g_1(\xi)<g_2(\xi)$ for $\xi\in(0,\xi_0)$ and $g_1(\xi_0)=g_2(\xi_0)$. Assume for contradiction that at $\xi=\xi_0$ both functions are still increasing, thus $g_1'(\xi_0)>0$, $g_2'(\xi_0)>0$ and $g_2'(\xi_0)\leq g_1'(\xi_0)$. It follows from the identity \eqref{int.ident} that
$$
g_2'(\xi_0)^2-g_1'(\xi_0)^2=g_2'(0)^2+\sigma\int_0^{\xi_0}\xi^{\sigma-1}(g_2(\xi)^2-g_1(\xi)^2)\,d\xi>0,
$$
which is a contradiction. Thus, either $g_1'(\xi_0)<0$ or $g_2'(\xi_0)<0$, and the corresponding function whose derivative is negative at $\xi=\xi_0$ attained its first local maxima at some smaller $\xi_M<\xi_0$. This proves that the graph of this function crossed first the hyperbola \eqref{hyp1} before intersecting the graph of the other function, as claimed.

\medskip

\noindent \textbf{Step 3. Monotonicity when condition (b) is fulfilled.} With the same notation as in Step 2, we first notice from \eqref{interm17} that
$$
(g_2-g_1)''(0)=\frac{1}{m-1}\left[g_2(0)^{1/m}-g_1(0)^{1/m}\right]>0,
$$
thus $g_2''(\xi)>g_1''(\xi)$ in a right-neighborhood of $\xi=0$. Let $\xi_0$ be the first intersection point of $g_2$ and $g_1$ as in Step 2, and $\xi_1$ be the first intersection point of their second derivatives $g_1''$ and $g_2''$, that is, $g_1''(\xi)<g_2''(\xi)$ for $\xi\in(0,\xi_1)$ and $g_1''(\xi_1)=g_2''(\xi_1)$. Since $g_1(\xi_0)=g_2(\xi_0)$, $g_1(\xi)<g_2(\xi)$ for $\xi\in(0,\xi_0)$ and $g_2'(0)=0=g_1'(0)$, which implies that $g_1'(\xi)<g_2'(\xi)$ for $\xi\in(0,\xi_1)$, it necessarily follows that $\xi_1\in(0,\xi_0)$ and $g_2(\xi_1)>g_1(\xi_1)$. Introducing the function
$$
\phi(x):=\frac{1}{m-1}x^{1/m}-\xi_1^{\sigma}x,
$$
it follows that $\phi(g_2(\xi_1))=\phi(g_1(\xi_1))$. Studying the variation of $\phi$, we notice that it has a maximum at
$$
x_0:=(m(m-1))^{-m/(m-1)}\xi_1^{-m\sigma/(m-1)}
$$
and it is increasing for $x\in(0,x_0)$ and decreasing for $x\in(x_0,\infty)$. It follows that $g_2(\xi_1)>x_0>g_1(\xi_1)$, hence
\begin{equation}\label{interm18}
f_1(\xi_1)<x_0^{1/m}=(m(m-1))^{-1/(m-1)}\xi_1^{-\sigma/(m-1)}<f_2(\xi_1).
\end{equation}
In particular, we infer from the second inequality in \eqref{interm18} that the point $(\xi_1,f_2(\xi_1))$ lies on the graph of $f_2$ after the first intersection with the hyperbola \eqref{hyp1}. Since $\xi_0>\xi_1$ the conclusion follows.
\end{proof}

\subsection{Non-existence}\label{subsec.nonexist}

Having in our hands the preparatory Lemma \ref{lem.monot}, we are now ready to prove Theorem \ref{th.nonexist}. Let us first notice that the hyperbola \eqref{hyp1} converts in the phase space associated to the system \eqref{PSsyst2} into the line of equation $(m-1)Y+\sigma X=0$ inside the plane $\{Z=1/m\}$.
\begin{proof}[Proof of Theorem \ref{th.nonexist}]
\noindent \textbf{Step 1. Shooting from the vertical axis.} Consider the profile $h(\xi)$ contained in unique orbit going out of $P_2$ into the phase space associated to the system \eqref{PSsyst2}, according to Lemma \ref{lem.P2}. We readily infer from Lemma \ref{lem.monot} that for any good profile $f$ such that $f(0)=a>0$, $f'(0)=0$ (and even with positive slope at the origin $f'(0)>0$), the first intersection point of the graph of $f(\xi)$ with the hyperbola \eqref{hyp1} lies "above" (in the geometric sense on the hyperbola) the first intersection point of the profile $h(\xi)$ with the same hyperbola. Thus, we want to estimate the coordinate of this latter intersection point. Since by Lemma \ref{lem.P0P2global} the orbit going out of $P_2$ stays outside the cylinder of equation \eqref{cylinder}, we infer that along this orbit
$$
Y^2\geq\frac{2}{m+1}-\frac{1}{m}Z,
$$
hence, recalling that we are dealing only with the region $\{Z<1/m\}$, that is, before the first intersection with the hyperbola \eqref{hyp1}, we get
\begin{equation*}
\begin{split}
\frac{2\sqrt{m(m-1)}}{m-1}(h^{(m-1)/2})'(\xi)&\geq\sqrt{\frac{2}{m+1}-\frac{Z}{m}}\geq\sqrt{\frac{2}{m+1}-\frac{1}{m^2}}\\
&=\sqrt{\frac{(m-1)(2m+1)}{m^2(m+1)}},
\end{split}
\end{equation*}
whence by integration on $(0,\xi)$ and taking into account that $h(0)=0$,
\begin{equation}\label{interm19}
h(\xi)\geq\left[\frac{m-1}{2m}\sqrt{\frac{2m+1}{m(m+1)}}\xi\right]^{2/(m-1)}.
\end{equation}
We easily deduce from \eqref{interm19} that the first intersection of the graph of $h(\xi)$ with the hyperbola \eqref{hyp1} occurs at a coordinate $\xi\leq\xi_{+}$, where $\xi_{+}$ satisfies the equality
$$
\left[\frac{m-1}{2m}\sqrt{\frac{2m+1}{m(m+1)}}\xi_+\right]^{2/(m-1)}=\left(\frac{1}{m(m-1)}\right)^{1/(m-1)}\xi_+^{-\sigma/(m-1)},
$$
that is,
\begin{equation}\label{interm20}
\xi_+=\left[\frac{4m^2(m+1)}{(2m+1)(m-1)^3}\right]^{1/(\sigma+2)}.
\end{equation}
In conclusion, \eqref{interm20} and Lemma \ref{lem.monot} show that any good profile (with or without interface) intersects for the first time the hyperbola \eqref{hyp1} at a point $\xi\leq\xi_+$.

\medskip

\noindent \textbf{Step 2. Shooting backward from the interface point.} Recall that the hyperbola \eqref{hyp1} is mapped in the phase space into the line of equations $\{(m-1)Y+\sigma X=0, \ Z=1/m\}$. Noticing that the direction of the flow on the plane $\{Z=1/m\}$ is given by the sign of the expression $(m-1)Y+\sigma X$, we infer that for
$$
X\geq x_0:=\frac{(m-1)h_0}{\sigma},
$$
the above flow has positive sign at points lying inside the cylinder, that is with $|Y|\leq h_0$. As we regard the orbits backward from the interface point, we deduce that the last intersection of an orbit with interface with the hyperbola \eqref{hyp1} (that is, the closest intersection to the vertical axis) occurs with $X\leq x_0$, that is,
$$
f(\xi)\leq\left[\frac{(m-1)h_0}{\sigma}\xi\right]^{2/(m-1)}.
$$
Thus, this last (in backward sense) intersection point with the hyperbola \eqref{hyp1} of a profile with interface occurs at a coordinate $\xi\geq\xi_{-}$, where $\xi_{-}$ satisfies the equality
$$
\left[\frac{(m-1)h_0}{\sigma}\xi_{-}\right]^{2/(m-1)}=\left(\frac{1}{m(m-1)}\right)^{1/(m-1)}\xi_-^{-\sigma/(m-1)},
$$
that is,
\begin{equation}\label{interm21}
\xi_{-}=\left[\frac{(m+1)\sigma^2}{2m(m-1)^3}\right]^{1/(\sigma+2)}
\end{equation}
Gathering \eqref{interm20} and \eqref{interm21}, we notice that for any $\sigma>0$ such that $\sigma^2>8m^3/(2m+1)$, we have
$$
\frac{(m+1)\sigma^2}{2m(m-1)^3}>\frac{4m^2(m+1)}{(2m+1)(m-1)^3},
$$
whence $\xi_{-}>\xi_+$ and the two shooting processes cannot join at the same point in order to form a complete good profile with interface. We thus conclude that at least for $\sigma>2m\sqrt{2m/(2m+1)}$ there is no good profile with interface, ending the proof.
\end{proof}

\noindent \textbf{Remarks.} 1. The above proof gives us a quantitative estimate for $\sigma$ such that non-existence of good profiles with interface occurs. However, numerical experiments seem to show that the inferior limit of non-existence for $\sigma$ is far below $2m\sqrt{2m/(2m+1)}$ (for example in Figure \ref{fig2} we perform the experiment for $m=2$ with $\sigma=1/2$) as the fact that a good connection between the vertical axis and an interface point is theoretically possible to exist does not necessarily mean that it really exists.

\medskip

\noindent 2. The proof of Theorem \ref{th.nonexist} together with the monotonicity Lemma \ref{lem.monot} prove that in fact all the profiles with interface for $\sigma>2m\sqrt{2m/(2m+1)}$ cut the vertical axis with $f'(0)<0$. Let us recall here the profiles plotted in Figure \ref{fig2} which confirm this behavior. It also appears from the numerical experiment in Figure \ref{fig2} that the profiles with interface seem to have their intersections with the vertical axis concentrated in a small region of this axis. We have no rigorous proof yet for such a statement.

\section*{Acknowledgements} R. I. is partially supported by the ERC Starting Grant GEOFLUIDS 633152. A. S. is partially supported by the Spanish project MTM2017-87596-P.

\bibliographystyle{plain}

\end{document}